\documentclass[11pt]{amsart}
\usepackage[latin1]{inputenc}
\usepackage{epsfig}
\usepackage{color}
\usepackage{amsthm}
\usepackage{amsmath}
\usepackage{amsfonts}
\usepackage{amssymb}
\usepackage{graphicx}
\usepackage[tight,FIGTOPCAP]{subfigure}
\usepackage[numbers,sort&compress]{natbib}
\usepackage{graphicx}
\usepackage[tight,FIGTOPCAP]{subfigure}
\usepackage{amsmath}
\usepackage{amsthm}
\usepackage{amscd}
\usepackage{amsfonts}
\usepackage{amssymb}
\usepackage[numbers,sort&compress]{natbib}
\newtheorem{corollary}{Corollary}[section]
\newtheorem{definition}[corollary]{Definition}

\newtheorem{lemma}[corollary]{Lemma}

\newtheorem{remark}[corollary]{Remark}

\newtheorem{theorem}[corollary]{Theorem}
\newfont{\sBlackboard}{msbm10 scaled 900}

\newcommand{\mylabel}[1]{\label{#1}
            \ifx\undefined\stillediting
            \else \fbox{$#1$}\fi }
\newcommand{\BE}{\begin{equation}}

\newcommand{\EEQ}{\end{equation}}
\newcommand{\rfb}[1]{\mbox{\rm
   (\ref{#1})}\ifx\undefined\stillediting\else:\fbox{$#1$}\fi}
\newcommand{\half}   {{\frac{1}{2}}}

\newfont{\Blackboard}{msbm10 scaled 1200}
\newcommand{\bl}[1]{\mbox{\Blackboard #1}}
\newfont{\roma}{cmr10 scaled 1200}

\def\CC{\rm \hbox{C\kern-.56em\raise.4ex
         \hbox{$\scriptscriptstyle |$}\kern+0.5 em }}
\newcommand{\be}{\begin{equation}}
\newcommand{\ee}{\end{equation}}
\newcommand{\beq}{\begin{eqnarray}}
\newcommand{\eeq}{\end{eqnarray}}
\newcommand{\beqs}{\begin{eqnarray*}}
\newcommand{\eeqs}{\end{eqnarray*}}
\newcommand{\bt}{\begin{Theorem}}
\newcommand{\et}{\end{Theorem}}
\newcommand{\br}{\begin{remark}}
\newcommand{\er}{\end{remark}}
\newcommand{\bc}{\begin{Corollary}}
\newcommand{\ec}{\end{Corollary}}
\newcommand{\el}{\end{Lemma}}
\newcommand{\bd}{\begin{definition}}
\newcommand{\ed}{\end{definition}}

\newcommand{\nline}  {{\bl N}}
\newcommand{\rline}  {{\bl R}}

\newcommand{\mm}    {{\hbox{\hskip 0.5pt}}}

\newcommand{\nm}    {{\hbox{\hskip -3pt}}}

\newcommand{\bluff} {{\hbox{\raise 15pt \hbox{\mm}}}}
\makeatletter
\def\section{\@startsection {section}{1}{\z@}{-3.5ex plus -1ex minus
    -.2ex}{2.3ex plus .2ex}{\large\bf}}
\makeatother
\def\be{\begin{equation}}
\def\ee{\end{equation}}

\def\ds{\displaystyle}
\begin{document}

\thispagestyle{empty}
\title[Boundary stabilization by a switching time-delay: a numerical study]{Boundary stabilization of a one-dimensional wave equation by a switching time-delay: a theoretical and numerical study}

\author{Ka\"{\i}s Ammari}
\address{UR Analysis and Control of PDEs, UR13ES64, Department of Mathematics, Faculty of Sciences of Monastir, University of Monastir, 5019 Monastir, Tunisia}
\email{kais.ammari@fsm.rnu.tn}

\author{Boumedi\`ene Chentouf}
\address{Kuwait University, Faculty of Science, Department of Mathematics, Safat 13060, Kuwait}
\email{boumediene.chentouf@ku.edu.kw; chenboum@hotmail.com}

\author{Nejib Smaoui}
\address{Kuwait University, Faculty of Science, Department of Mathematics, Safat 13060, Kuwait}
\email{n.smaoui@ku.edu.kw; nsmaoui64@yahoo.com}

\date{}

\begin{abstract}
This paper deals with the boundary stabilization problem of a one-dimensional wave equation with a switching time-delay in the boundary. We show that the problem is well-posed in the sense of semigroups theory of linear operators. Then, we provide a theoretical and numerical study of the exponential stability of the system under an appropriate delay coefficient.
\end{abstract}


\maketitle


\thispagestyle{empty}

\section{Introduction}
\setcounter{equation}{0}
This article deals with the boundary stabilization of the following switching time--delay wave equation in $(0,\ell)$:
\begin{eqnarray}
& & y_{tt}(x,t) - y_{xx}(x,t)=0, \hspace{6.8 cm} (x,t) \in (0,\ell) \times
(0,+\infty),\label{1.1}\\
& & y(0,t)  =0,\quad \quad \hspace{8.2 cm} t \in
(0,+\infty),\quad  \label{1.2}\\
& & y_x(\ell,t)  =0, \quad \quad \hspace{8.1 cm} t \in (0,2\ell), \label{1.2bis}\\
& & y_x(\ell,t)= \mu y_t(\ell, t-2\ell),  \quad \quad \hspace{6.2 cm}
t \in (2\ell,+ \infty), \label{dampdelay}\\
& &y(x,0)=y_0(x), \; y_t(x,0)=y_1(x), \quad \quad \hspace{4.7 cm} x \in (0,\ell),\label{1.4}
\end{eqnarray}
where $\ell > 0$ and $\mu \in \mathbb{R}$.

\medskip

It is well-known that the presence of a time--delay  is usually unavoidable in practice. In fact, such a phenomenon arises in many applications for different reasons and from numerous  sources. Furthermore, it has been noticed that even an arbitrarily small delay may have a destabilizing effect on an originally stable system, that is, some systems are stable in the absence of a time--delay but then become unstable as long as a delay is taken into consideration (see for instance \cite{Datko, DLP, Datko97} and \cite{NPSicon}). It turned out that unstable system are unexpectedly stabilized under the action of a ``well chosen'' time--delayed control \cite{bas, 3j, g, gt, guo, cer}, while other systems are not affected at all in the sense that the presence of a delay does have any impact on the stability property \cite{guoxu}. This may explain the very active and intensive research studies on the delay effect in the stabilization of  systems. For instance, stability results are available in literature for systems with time--delays whose negative impact is neutralized due to the presence of appropriate feedback controls  (see e.g. \cite{amman, ac1, ac2, ac3, ANPi, NPSicon, NVCOCV10}).

\medskip

Going back to the system under consideration, let us note that the term $\mu y_t(\ell,t- 2 \ell)$ can be viewed as a switched boundary control. Clearly,  such a feedback is unbounded. On the other hand, the system \rfb{1.1}--\rfb{1.4} has been shown to be exponentially stable if $\mu <0$ and in the absence of time--delay \cite{Chen}.

\medskip

It is also worth noting that in \cite{anp}, the authors considered the following system with a switching between the damping and delay,
\begin{eqnarray*}
& & y_{tt}(x,t) - y_{xx}(x,t)=0\quad \mbox{\rm in}\quad(0,\ell) \times
(0,+\infty),\\
& & y(0,t)  =0\quad \mbox{\rm on}\quad
(0,+\infty),\quad  \\
& & y_x(\ell,t)  =0\quad \mbox{\rm on}\quad
(0,2\ell),\quad  \\
& & y_x(\ell,t)= \mu_1 y_t(\ell,t)  \quad
\mbox{\rm on}\quad
(2(2i+1) \ell,2(2i +2)\ell),~ \, \, i \in \mathbb{N}, \\
& & y_x(\ell,t)= \mu_2 y_t(\ell, t-2\ell)  \quad
\mbox{\rm on}\quad
(2(2i+2)\ell,2(2i + 3)\ell),~ \, \, i \in \mathbb{N}, \\
& &y(x,0)=y_0(x)\quad \mbox{\rm and}\quad y_t(x,0)=y_1(x)\quad
\hbox{\rm in} \quad(0,\ell),
\end{eqnarray*}
where $\ell > 0$ and $\mu_1, \mu_2$ are constants. Then, it has been shown in \cite{anp} that the above system is exponentially stable provided that  $1 < \mu_2 < \mu_1$ or $\mu_1 < \mu_2 < 1$. However, in the present paper, we consider a delay on $(2\ell,+\infty)$ and conduct a theoretical as well as  numerical investigation of the exponential stability of the system (\ref{1.1})-(\ref{1.4}) in terms of the parameter $\mu$.

\medskip

The main contribution of this paper is to show the well-posedness of the problem \rfb{1.1}--\rfb{1.4} and carry on a theoretical and numerical study of the boundary stabilization of the switching delay wave system \rfb{1.1}--\rfb{1.4}. Indeed, we first establish the existence and uniqueness of solutions to the problem under consideration. Subsequently, we theoretical and numerical study the wave system \rfb{1.1}--\rfb{1.4} when one stabilizes it by a control law that uses information from the past, either by  switching or not. In other words, the stabilization outcome is established  by a control method \cite{ANPi,AG} contrary to  the methodology which is based on a feedback law. As pointed out in \cite{anp}, the control method may help to design a time--delay compensation scheme. This is the typical  predictive controller for systems with pure time--delay control, which  evokes a feedback loop to control any system. In such a situation, the predictor control is designed with the aim of eliminating any  effect of the time on the closed--loop \cite{smith, smithbis}. On the other hand, the novelty of our work compared to the previous ones \cite{anp, g, gt} is threefold: first, our switching delay occurs on the whole time interval $(2\ell, \infty)$, while it is just on $  (2(2i+2)\ell,2(2i + 3)\ell),~ \, \, i \in \mathbb{N}$ in \cite{anp}. Second, we are able to enlarge the admissibility interval of $\mu$ to $(0,1)$ instead of $((\sqrt{2}-1)^2 ,1)$ (resp. $I \subset (-2/3,0)$) imposed in \cite{g} (resp. \cite{gt}). On the other hand, contrary to \cite{gt} where the delay belongs to $\{4/{\ell}, 8/{\ell}\}$, our delay occurs at $2 \ell$, which corresponds to the optimal time of the response of the system and more importantly the optimal time of the observability. Third, numerical simulations are provided to support and ascertain the validity of our theoretical outcomes.

\medskip

The paper is organized as follows. In the second section, we show that the problem \eqref{1.1}--\eqref{1.4} is well-posed and the solutions are indeed exponentially stable. In the third section, we numerically show the exponential stability of the  system \eqref{1.1}--\eqref{1.4}.

\section{Well-posedness and stability of the problem}
\setcounter{equation}{0}

\subsection{Well-posedness of the problem}
This subsection is aimed to set our system \rfb{1.1}--\rfb{1.4} in an appropriate  functional space. This will enables us to state and prove the existence and uniqueness of solutions to \rfb{1.1}--\rfb{1.4} by means of semigroups theory of linear operators. To do so, let $A = - \partial^2_x$ be the unbounded operator in $V = L^2(0,\ell)$
with domain
$$V_{1}={\mathcal D}(A) = \left\{f \in H^2(0,\ell); \, f(0) = 0, \, f_x(\ell) = 0 \right\},$$
whereas
$$
V_\half = {\mathcal D}\left(A^\half\right) = \left\{f \in H^1(0,\ell); \, f(0) =0 \right\}.$$
Subsequently, we define the operator $P \in {\mathcal L}(\rline, V_{-\half})$ as follows:
$$  Pr = \sqrt{\mu} \, A_{-1} N r = r \sqrt{\mu} \delta_\ell, \forall \, r \in \rline, \quad \text{and} \quad  P^*y = \sqrt{\mu} \, y(\ell), \, \forall \, y \in V_\half,
$$
in which $A_{-1}$ is the extension of $A$ to $V_{-1} = \left({\mathcal D}(A)\right)^\prime$ and $\delta$ is the Dirac mass at $\ell$. Furthermore, $N$ is the Neumann map defined by $\partial_x^2(N r) = 0$ on $(0,\ell)$ and $N r(0) = 0, \partial_x(Nr)(\ell) = r$. Finally, $V_{-\half} = \left(V_\half \right)^\prime$ (the duality is in the sense of $V$).

\medskip

The ultimate objective is to write  the system \rfb{1.1}--\rfb{1.4} as a differential equation in an appropriate functional space and then invoke semigroups theory. To this end, consider the Hilbert state space
$${\mathcal V}~:=~ V_{\half} \times V.$$
Then, we define the unbounded linear operator
\be \label{op}
\left\{
\begin{array}{l}
{\mathcal  A}: {\mathcal D}({\mathcal A}) =  V_{1} \times V_{\half} \subset {\mathcal  V} \longrightarrow {\mathcal  V}, \\[2mm]
{\mathcal A} \left( \begin{array}{ccc} y \\ z\end{array}\right)
= \left(
\begin{array}{lc}
& z \\
- & Ay
\end{array}
\right), \quad \text{for each} ~ (y,z) \in {\mathcal D}({\mathcal A}).
\end{array}
\right.
\ee
The operator ${\mathcal  A}$ defined by \eqref{op} generates strongly a group of isometries $({\mathcal S}(t))_{t\in \mathbb{R}}$ in ${\mathcal V}$. We shall also denote $({\mathcal S}_{-1}(t))_{t\in \mathbb{R}}$ the extension of $({\mathcal S}(t))_{t\in \mathbb{R}}$ to $(\mathcal{D}(\mathcal{A}))^\prime := V \times V_{- \half}$.

\medskip

We have the following result:

\begin{theorem}\label{propexistunic}
The system \rfb{1.1}--\rfb{1.4} is well-posed. Specifically, for every initial date $(y_0,y_1)\in {\mathcal  V}$, the solution of \rfb{1.1}--\rfb{1.4} has the following explicit form:
$$
\left(
\begin{array}{ccc}
y(t)\\
y_t(t)
\end{array}\right)= \left\{
\begin{array}{ll}
\left(
\begin{array}{ccc}
y^0(t)\\
y_t^0(t)
\end{array}\right)
= {\mathcal S}(t)\left(
\begin{array}{ccc}y_0\\y_1\end{array}\right), \quad  t \in [0, 2 \ell], \\[5mm]
\left(
\begin{array}{ccc}
y^{j}(t)\\ y_t^{j}(t)
\end{array}\right) = {\mathcal S}(t-2j \ell)\left(
\begin{array}{ccc}y^{j-1}(2j \ell)\\y_t^{j-1}(2j \ell)\end{array}\right) +
\ds \int_{2j \ell}^{t} {\mathcal S}_{-1}(t-s)\left(
\begin{array}{ccc}0\\- \mu \, y_t^{j-1}(s- 2 \ell) \delta_\ell\end{array}\right) \, ds, \, \\[5mm]
\hspace{7. cm}  t \in [2j \ell, 2 (j+1) \ell], \quad  j \in \nline^*=\nline \setminus \{0\}.
\end{array}
\right.
$$
Moreover, we have $$(y^{j},y^{j}_t) \in C([2j\ell, 2(j+1)\ell], {\mathcal  V}), \quad \text{for} \, j \in \nline.$$
\end{theorem}

\begin{proof}

Before proving this result, we shall  define the energy corresponding to a solution of the system \rfb{1.1}--\rfb{1.4} as follows:
\begin{equation}\label{energy}
\hspace{3cm}
E(t) =\frac{1}{2}\int_0^\ell \left[  y_x(x,t)^2+  y_t(x,t) ^2 \right] dx.
\end{equation}
Then, consider the evolution problems
 \be
\label{OPEN1}
u_{tt} ^j(t) + A u^j(t)  \,
=  Pz^j(t), \,\hbox{ in } (2j\ell,2(j+1) \ell), \, j \in \nline^*,
\ee
\be
\label{eq3}
u^j(2j\ell)= u_t^j(2j\ell) = 0, \, j \in \nline^*.
\ee

\be
\label{eq4}
\phi_{tt}(t) +
 A\phi(t)  \,
=  0, \, \hbox{ in }  (0,+\infty),
\ee
\be
\label{OPEN2}
\phi(0)= \phi_0, \, \phi_t(0) = \phi_1.
\ee

Subsequently, we shall investigate the regularity of $u^j$ when $z^j \in L^2(2\ell j,2(j+1)\ell j), \, j \in \nline^*$. Using  standard energy estimates, one can readily verify that
$$y^j \in C\left([2j\ell,2(j+1)\ell];V\right) \cap C^1\left([2j\ell,2(j+1)\ell];V_{- \half}\right).$$
In turn, in the case when the operator
$P$ satisfies an admissibility condition, then the solution $y^j$ has more regularity. Indeed, we have the following result, which is a special case of the general transposition method \cite{linsmag}).

The system \rfb{eq4}--\rfb{OPEN2} has a unique solution $\phi$ such that
$$
\phi \in C([0,2 \ell];V_\half) \cap C^{1}([0,2  \ell];V),
$$
$$
(\phi,\phi_t)(t) = {\mathcal S}(t)\left(
\begin{array}{ccc}\phi_0\\\phi_1\end{array}\right), \quad
\forall t \in [0, 2  \ell].
$$
Furthermore, we have
$P^* \phi(\cdot)\in H^1(0,2 \ell), $ and for all $T \in (0, 2\ell)$, there exists a constant
$C>0$
such that
\be
\Vert (P^* \phi)^{\prime}(\cdot)\Vert_{L^2(0,T)}\le C \, \Vert (\phi_0,\phi_1)\Vert_{
V_\half \times V}, \quad \text{for all} \; (\phi_0,\phi_1)\in V_\half \times V.
\label{CACHE1}
\ee

\begin{lemma} \label{ex}
Assume that $z^j \in L^2([2j\ell,2(j+1)\ell]), \, j \in \nline^*$.
Then, the system \rfb{OPEN1}--\rfb{eq3} posses a unique solution with the following  regularity
\be
u^j \in C([2j\ell,2 (j + 1) \ell];V_\half) \cap C^{1}([2\ell j,2 (j + 1) \ell];V), \forall  j \in \nline^*.
\label{REG1}
\ee
Moreover, we have
$$
(u^{j},u_t^{j})(t) =   \ds \int_{2j\ell}^{t} {\mathcal S}_{-1}(t-s)\left(
\begin{array}{ccc}0\\P v_j(s)\end{array}\right) \, ds, \quad
\forall t \in [2j \ell ,  2 (j+1) \ell],\  \forall j \in \nline^*.
$$
\end{lemma}
\begin{proof}
Let
$$W(t) = \left(
\begin{array}{ll}
u^j(t+2j\ell) \\
\nm
\ds
u_t^j(t+2j\ell)
\end{array}
\right).$$
Thereafter, the system \rfb{OPEN1}--\rfb{eq3} can be formulated as follows
\[
W_t^j + {\mathcal A} W^j(t) = {\mathcal P} z^j(t+2j\ell) \hbox{ on } (0, 2\ell), \, W^j(0) = 0,
\]
where
\[
{\mathcal A} = \left(
\begin{array}{cc}
0 & - I \\
\nm
\ds
A &0
\end{array}
\right): V_\half \times V \rightarrow [{\mathcal D}({\mathcal A})]^{\prime} ,
\]
\[
{\mathcal P} = \left(
\begin{array}{ll}
0 \\
\nm
\ds
P
\end{array}
\right): \rline \rightarrow [{\mathcal D}({\mathcal A})]^{\prime}.
\]
Since the operator ${\mathcal A}$ is skew adjoint operator, it generates a group of isometries ${\mathcal T}(t)$ in $[{\mathcal D}({\mathcal A})]^{\prime} $. Clearly, ${\mathcal T}(t) = {\mathcal S}_{-1}(t)$.

It is easy to check that the operator ${\mathcal P}^* : {\mathcal D}({\mathcal A}) \rightarrow \rline$ is given by
\[
{\mathcal P}^*  \left(\begin{array}{c}
u^j \\
z^j
\end{array}
\right) = P^* z^j , \,\text{for all} \,
(u^j,z^j) \in  {\mathcal D}({\mathcal A}).
\]
Whereupon
\[
{\mathcal P}^*{\mathcal T}^*(t)
\begin{pmatrix} \phi_0 \cr \phi_1 \end{pmatrix} = P^* \phi_t(t), \, \text{for all} \,
(\phi_0,
\phi_1) \in  {\mathcal D}({\mathcal A}),
\]
where $\phi$ is a solution of \rfb{eq4}--\rfb{OPEN2}. This together with \rfb{CACHE1} imply that there exists a constant $C > 0$ such that for all $T\in (0, 2\ell)$
\[
\int_{0}^{T}  \left| {\mathcal P}^*{\mathcal T}^*(t) \left(\begin{array}{c}
\phi_0 \\
\phi_1
\end{array}
\right) \right|^2 \, dt \le C \, ||(\phi_0,\phi_1)||^2_{V_\half \times V}, \, \text{for all} \,
(\phi_0,\phi_1) \in  {\mathcal D}({\mathcal A}).
\]
Invoking Theorem 3.1 in \cite[p.187]{ben} (see also \cite{tucsnakweiss}), the above estimate implies the required interior regularity  \rfb{REG1}.
\end{proof}

Now, it suffices to use induction in order to establish the existence result for the system \rfb{1.1}--\rfb{1.4}. To proceed, we set on $[0,2\ell]$ (case   $j=0$):
$$
\left(
\begin{array}{ccc}
y^0(t)\\ y^0_t(t)
\end{array}\right) = {\mathcal S}(t)\left(
\begin{array}{ccc}y_0\\y_1\end{array}\right), \, \text{for all} \, t \in [0,2\ell].$$
It is clear that the above definition provides a solution of \rfb{eq4}--\rfb{OPEN2} on $(0,2\ell)$, with the regularity $(y^0,y^0_t) \in C([0,2\ell];{\mathcal V})$. Next, if $j\geq 1$, then we define for all $t\in [2j \ell ,2 (j+1) \ell]$
$$
\left(
\begin{array}{ccc}
y^{j}(t) \\ y_t^{j}(t)
\end{array}
\right) =
\left(
\begin{array}{ccc}
\phi(t+2j\ell) \\ \phi_t(t+2j\ell)
\end{array}\right) +
\left(
\begin{array}{ccc}
u^{j}(t)\\ u_t^j(t)
\end{array}
\right)
$$
$$
={\mathcal S}(t+2j\ell)\left(
\begin{array}{ccc}y^{j-1}(2j\ell) \\ y_t^{j-1}(2j\ell)\end{array}\right)   +\ds \int_{2j\ell}^{t} {\mathcal S}_{-1}(t-s)\left(
\begin{array}{ccc}0\\- \mu \, y_t^{j-1}(\ell,s - 2 \ell) \delta_\ell\end{array}\right) \, ds,
$$
where $u^j$ (resp. $\phi$)  is the solution of \rfb{OPEN1}--\rfb{eq3} (resp. \rfb{eq4}--\rfb{OPEN2}) and
$$z^j(t) = - \mu \, y^{j-1}_t(t - 2 \ell).$$
The latter belongs to $L^2(2j\ell,2(j+1)\ell)$ since the operator $P^*$ is an input admissible operator (see \cite{aht} for more details) and as  $\phi_0 = y^{j-1}(2j\ell),$  $\phi_1 = y_t^{j-1}(2j\ell)$. This argument allows us to claim that such a solution has the desired regularity.
\end{proof}

\subsection{Stability of the system}

We have the following stability result:

\begin{theorem} \label{princ}
For any $\mu \in (0, 1)$, there exist two  positive constants $M$ and $\omega$ such that for all initial data in ${\mathcal V}$, the energy $E(t)$, corresponding to the  solution of problem \eqref{1.1}-\eqref{1.4}, satisfies
\begin{equation}\label{expestimate}
E(t) \le M \, e^{- \, \omega t}, \quad \forall t \geq 0.
\end{equation}
The constant $M$ depends  on the initial data, on $\ell$ and $\mu$, whereas the rate $\omega$ solely depends on the physical parameters $\ell$ and $\mu$.
\end{theorem}
\begin{proof}
First, we seek the solution  $u$ of the system (\ref{1.1})--(\ref{1.4})  in the
form:
\begin{equation}\label{defu}
u(x,t)=\Theta(x+t)-\Theta(t-x),  \quad x \in (0,\ell),\  t\geq 0,
\end{equation}
where $\Theta$ is a function to be found in $H^1_{\rm loc} (-\ell,+\infty)$.

In light of (\ref{defu}), the boundary condition (\ref{1.2}) clearly holds. Thereafter, the initial conditions (\ref{1.4}) are satisfied by taking
\beqs
\Theta(x)&=&-\frac12
u_0(-x)+\frac12\int_0^{-x} u_1(s)ds \quad \text{for all} \,
x\in (-\ell,0),\\
\Theta(x)&=&\frac12 u_0(x)+\frac12\int_0^x u_1(s)ds \quad \text{for all} \,
x\in [0,\ell). \eeqs
Subsequently, (\ref{1.2bis}) holds when
$$
\Theta'(\ell+t)+\Theta'(t-\ell)=0, \hbox{ for } 0<t<2\ell,
$$
that is,
$$
\Theta'(y)=-\Theta'(y-2\ell ) \, \text{for all} \, y \in (\ell, 3\ell).
$$
Whereupon, the existence of $\Theta$ on $(\ell, 3\ell)$ follows as we know the right-hand side.

In turn, (\ref{dampdelay})  is fulfilled whenever
$$
\Theta'(\ell+t)+\Theta'(t-\ell)=\mu (\Theta'(-\ell+t)-\Theta'(t-3\ell)), \hbox{
for } t\in (2\ell,+\infty),
$$
which can be rewritten as follows
\be
\label{DL}
\Theta'(y) = (-1+\mu)\Theta'(y-2\ell ) - \mu \Theta^\prime (y-4\ell), \text{for all} \, y \in
(3\ell,+\infty).
\ee
Next, we can show that $\Theta$ is well-defined on the whole interval
$(-\ell, \infty)$ by means of an induction argument.

Then, we can define, for $y>5\ell$, the vector
$$R(y):=\left(
\begin{array}{l}
\; \; \; \Theta'(y)\\
\Theta'(y-2\ell)
\end{array}
\right ),
$$
which, together with (\ref{DL}), implies that
$$
R(y)=G_\mu R(y-2\ell),
$$
where $G_\mu$ is the matrix
$$
G_\mu =\left(
\begin{array}{ll}
\mu-1 & - \mu\\
\; \; \; 1& \; 0
\end{array}
\right).
$$
Arguing as  in \cite{g, anp}, it suffices to compute the eigenvalues of the matrix $G_\mu$ whose characteristic polynomial is given by
$$
p_\mu(\lambda)=\lambda^2 + (1-\mu)\lambda +\mu.
$$
The roots of $p_\mu$ are given by
$$
\lambda =\frac{\mu-1\pm\sqrt{\mu^2-6 \mu+1}}{2},
$$
and hence the modulus of the eigenvalues of $G_\mu$ is strictly less than $1$ if and only if
\begin{equation}\label{cdspect}
|\mu-1\pm\sqrt{\mu^2-6 \mu+1}|<2.
\end{equation}
A simple calculation shows that if $\mu^2-6 \mu+ 1\geq 0$, then (\ref{cdspect}) holds
if and only if
\begin{equation}\label{cdspect1}
0<\mu\leq 3-2\sqrt{2}.
\end{equation}
However, in the case  $\mu^2-6 \mu+1 < 0$, then  (\ref{cdspect}) is valid
if and only if
\begin{equation}\label{cdspect2}
  3-2\sqrt{2} < \mu <1.
\end{equation}
Whereupon, (\ref{cdspect}) holds if and only if $\mu\in (0,1)$.

On the other hand, since
$$
p_\mu'(\lambda)= 2 \lambda + 1 - \mu,
$$
we can conclude that for  $\mu \in (0, 1)$,  the  eigenvalues of $G_\mu$ are  of modulus $<1$ and simple. In such an event, there exists an invertible matrix $V_\mu$ such that
$$
G_\mu=V_\mu^{-1} D_\mu V_\mu,
$$
where $D_\mu$ is the diagonal matrix made of the eigenvalues of $G_\mu$.

\medskip

Once again, using an inductive argument, we can deduce that for all $j\in \mathbb N$, and for all $y \in (5\ell + 2 j\ell, 5\ell + 2(j+1)\ell]$, we have
$$
Q(y)=G_\mu Q(y-2j\ell),
$$
in which
$$
Q(y):=\left(
\begin{array}{lll}
\; \; \; \Theta' (y)\\
\Theta' (y-2\ell)
\end{array}
\right).
$$
Combining the latter with the above factorization of $G_\mu$, we obtain
$$
Q(y)=V_\mu^{-1} D\mu^j V_\mu Q(y-2j\ell).
$$
This leads us to find a positive constant $C_\mu$, depending only on $\mu$, such that for all $j \in \mathbb N$, and  all $y \in (5\ell+2j\ell, 5\ell+2(j+1)\ell]$, we have

\begin{equation}\label{normeit}
 \|Q(y)\|_2\leq C_\mu \rho_\mu^j \|Q(y-2j\ell)\|_2,
\end{equation}
where $\rho_\mu$ is the spectral radius of $D_\mu$ that is $<1$ as long as $\mu \in (0,1)$.

\medskip

By virtue of (\ref{energy}) and (\ref{defu}), a simple computation yields
$$
E(t)=\int_{-\ell}^{\ell} \Theta' (x+t)^2 \,dx.
$$
Now, we use the same the arguments as in \cite{anp, g, gt} to conclude the exponential decay of the system. To proceed, for all $j\in \mathbb N$,
and for all $t \in (2\ell+ 2j\ell, 2\ell+ 2(j+1) \ell]$, one can apply (\ref{normeit}) with $y=x+t$, for any $x\in (-\ell,\ell)$. This implies that
\begin{eqnarray*}
E(t)&\leq& \int_{-\ell}^{\ell} \|Q(x+t)\|_2^2 \,dx\\
&\leq& C_\mu^2\rho_{\mu}^{2j}  \int_{-\ell}^{\ell} \|Q(x+t-2j\ell)\|_2^2 \,dx.
\end{eqnarray*}
Afterwards, using the fact that whenever $t \in (2\ell+2j\ell, 2\ell+2(j+1)\ell]$ and
 $x\in (-\ell,\ell)$, the quantity $x+t-2j\ell$ belongs to a compact set, and consequently one can conclude that
$$
\int_{-\ell}^{\ell} \|Q(x+t-2j\ell)\|_2^2 \,dx
$$
is bounded independently of $j$. Therefore, we managed to find
a constant $K_\mu$ such that for all $j\in \mathbb N$,
and for  all $t \in (2\ell+2j\ell, 2\ell+2(j+1)\ell]$, we have the estimate
$$
E(t) \leq  K_\mu \rho_\mu^{2j}.
$$
Lastly, the latter leads to the required result since
$\rho_\mu^{2j}=e^{2j\ln \rho_\mu}\leq \rho_\mu^{-4} \, e^{t\frac{\ln \rho_\mu}{\ell}}$.

\end{proof}

\section{Numerical study}\label{mainp}
\setcounter{equation}{0}

Numerical solutions for the one-dimensional wave equation (1.1)-(1.5) with and without a presence of a switching time-delay were simulated using COMSOL Multiphysics software. This software uses the finite element method (FEM) to approximate the partial differential equation and numerically finds its solutions. The solutions are computed for different values of $\mu$.

\indent  First, we consider the following wave equation \textbf{without} the presence of a switching time-delay:
\begin{eqnarray}
& &  y_{tt}(x,t) - y_{xx}(x,t)=0, \quad \hspace{7.2 cm} \;  (x,t) \in  (0,\ell) \times
(0,+\infty),\label{nod1}  \\
 & &y(0,t)  =0,  \quad \hspace{9.3 cm} \; t \in
(0,+\infty),   \label{nod2}\\
& & y_x(\ell,t)= \mu y_t(\ell, t),  \quad \hspace{8.1 cm}
t \in  (0,+ \infty), \label{nod3}  \\
& & y(x,0)=y_0(x)\quad \mbox{\rm and}\quad y_t(x,0)=y_1(x), \quad
\hspace{4.6 cm}
x \in (0,\ell),\label{nod4}
\end{eqnarray}
where  $\mu$ is a constant.

In this case, we take $\ell=1$ and $y_0(x)=y_1(x)=\sin(\pi x)$. Then, it has been observed that the dynamics of the system (\ref{nod1})-(\ref{nod4}) is exponentially stable if $\mu \, < 0$, and unstable if $\mu \, >0$. More precisely, Figure 1 depicts a 3-dimensional landscape of the dynamics of the above wave equation which indicates that the dynamics exponentially converges to the zero dynamics when $\mu \, < \, 0$. Figure 2 shows the energy $E(t)$, as defined by (\ref{energy}), versus time for different values of $\mu$. It is shown that the energy converges exponentially faster as the values of $\mu$ decreases from $\mu=-0.1$ to $\mu=-0.5$. On the other hand, Figure 3 shows that the dynamics of the wave system when $\mu > 0$ is unstable, and Figure 4 indicates that the corresponding energies for the dynamics presented in Figure 3 grow without bounds as the values of $\mu$ increases from $0.1$ to $0.5$. These results are in line with the theoretical findings established in \cite{Chen}.

Then, the wave system (1.1)-(1.5) \textbf{with} the presence of a switching time-delay is considered.  The system is simulated when $\ell=1$, $y_0(x)=y_1(x)=\sin(\pi x)$, and for different values of $\mu$. Figure 5 presents the dynamics of $y(x,t)$, and Figure 6 shows the energy, E(t), vs. time when  $\mu \, < \,0$. These figures indicate that the dynamics diverges faster as the value of $\mu$ decreases. Therefore, the presence of a time delay $\tau=2$ destabilizes a stable dynamics when $\mu$ is negative. This result is in accordance with our Theorem 2.3 and confirms the situation where a delay destabilizes a stable system as in [2, 16-18, 25].

In turn, if we let $ 0 < \mu < 1$ as in Section 2.3, and keep the time delay $\tau=2$, that is, $\ell=1$, then the dynamics of the controlled wave equation becomes stable (see Figures 7 and 8). The dynamics of the solution $y(x,t)$ to (1.1)-(1.5) are depicted in Figure 7, which shows the rapid stability of the solution. Furthermore, Figure 8a) shows that the energy, $E(t)$, decays exponentially faster as $\mu$ increases from 0.1 to 0.5, and Figure 8b) depicts that the energy decays sinusoidally at a slower rate from $\mu=0.8$ up to $\mu=0.95$, and grows at $\mu=1.0$. These results reinforce  the stability results shown in Section 2.3, and in accordance with the results in [12, 13, 20, 21, 23] that indicate that unstable systems can be stabilized under the action of a ``well chosen" time-delayed control. It should be noted that our numerical simulations indicate that  the switching time-delay with $\tau=2$ used to stabilize  the system (1.1)-(1.5) is not effective beyond $\mu=1$. That is for $\mu \ge 1$, the dynamics of the one-dimensional wave equation is unstable with or without the presence of a switching time-delay (see Figure 9).

\vspace{5mm}

\begin{figure}[!ht]
\centering
\subfigure[]{
\includegraphics[scale=0.5]{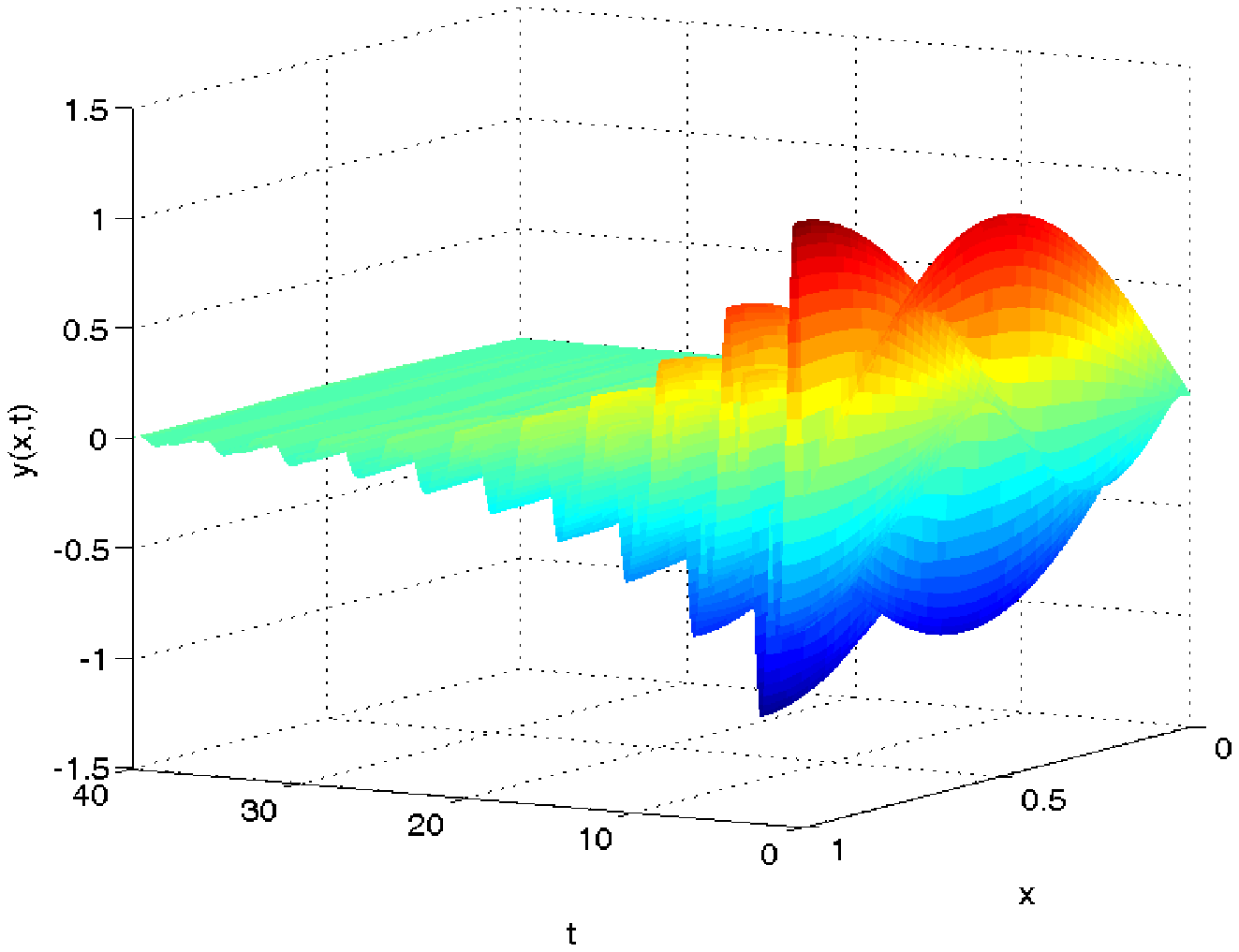}
\label{fig:subfiga1}
}
\subfigure[]{
\includegraphics[scale=0.5]{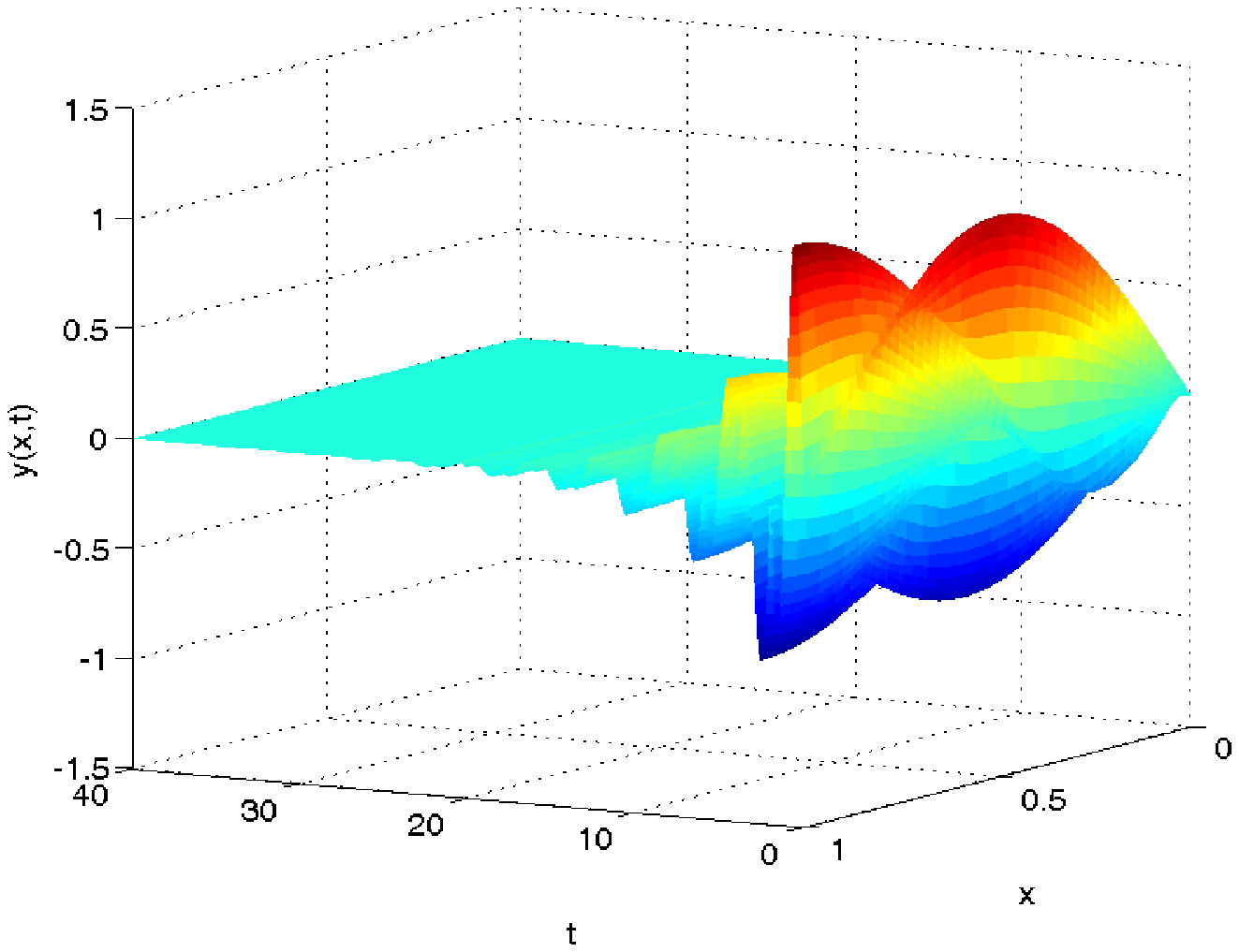}
\label{fig:subfigb1}
}
\subfigure[]{
\includegraphics[scale=0.5]{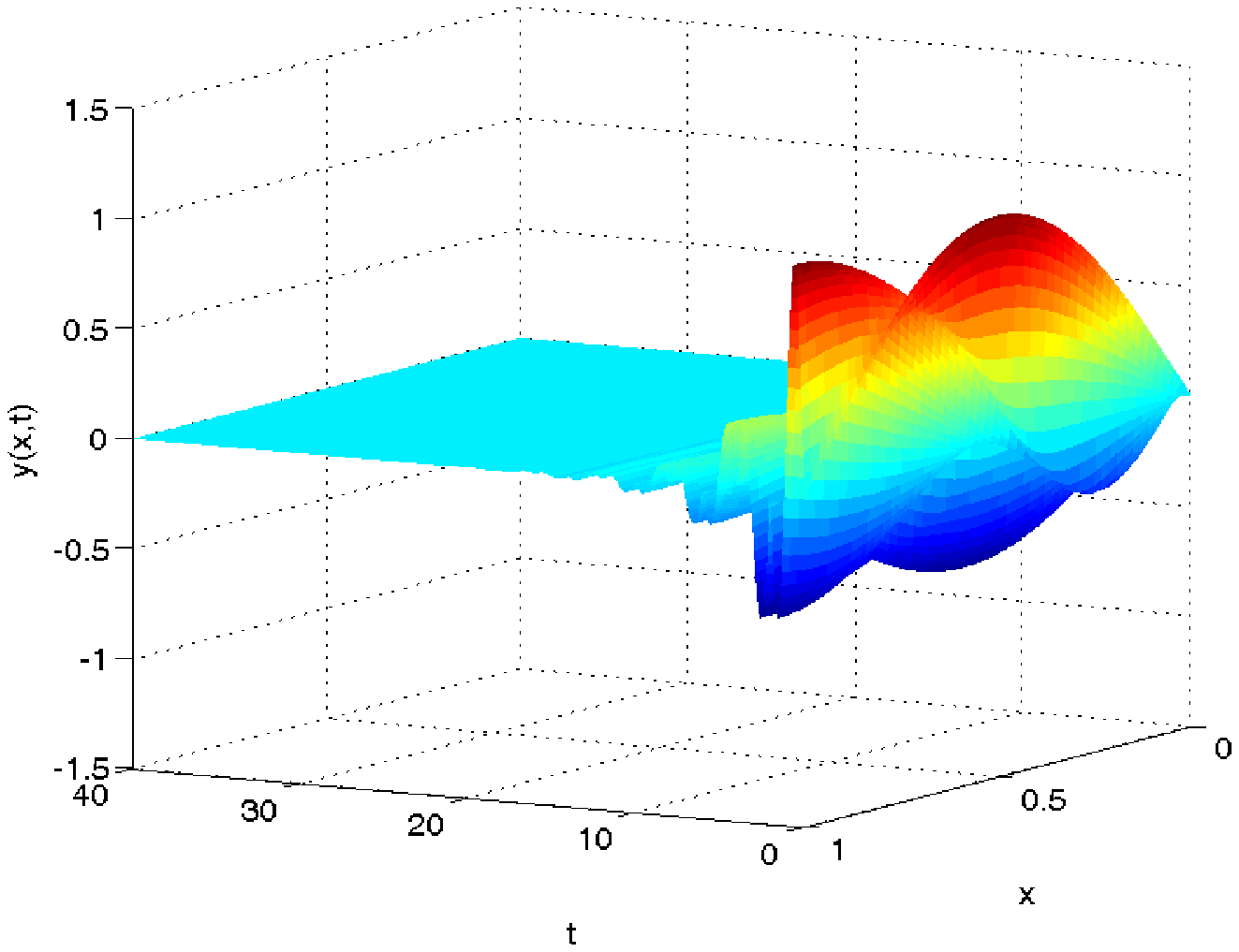}
\label{fig:subfigf1}
}
\subfigure[]{
\includegraphics[scale=0.5]{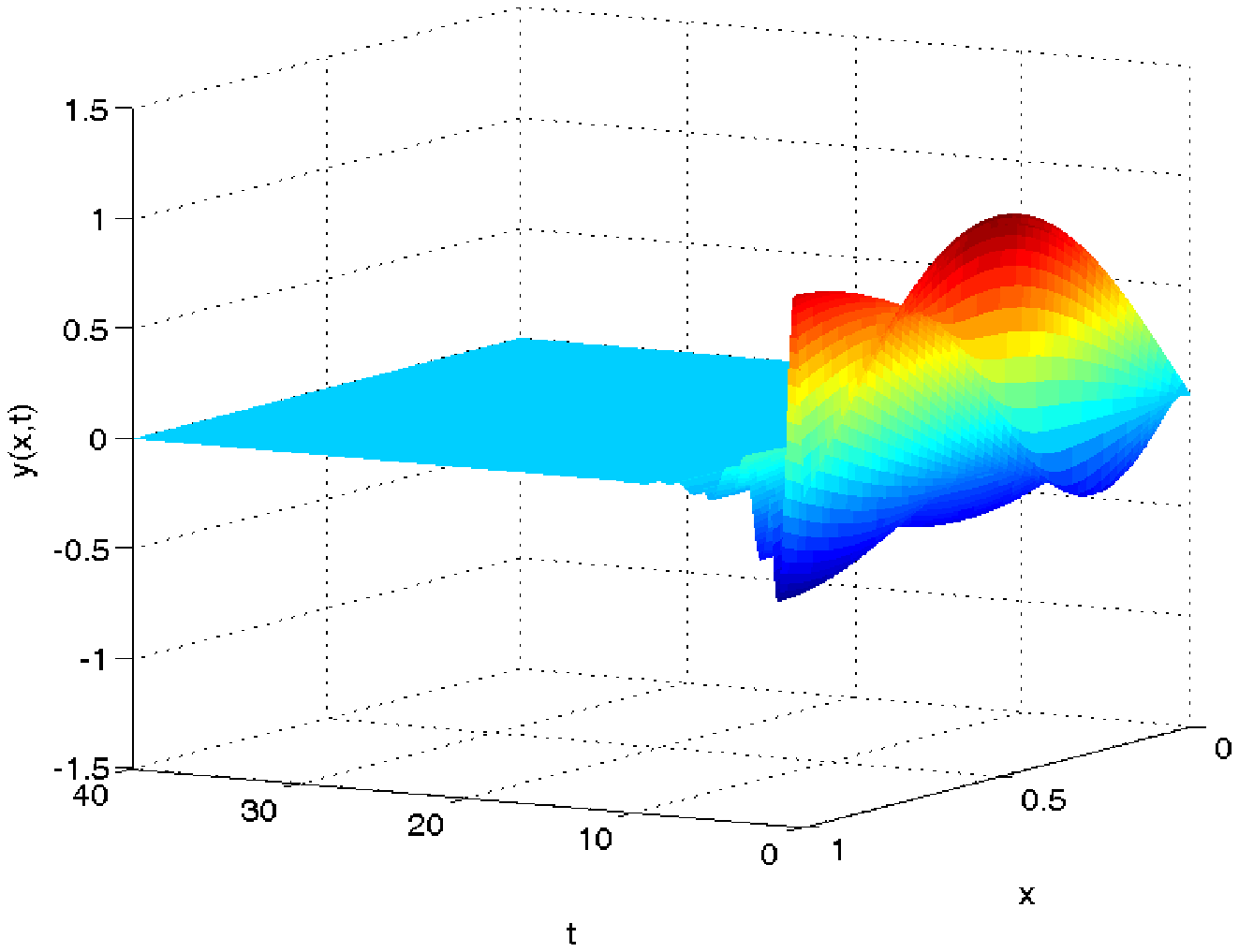}
\label{fig:subfige1}
}
\caption{A 3-d landscape of the dynamics of the wave equation without delay (\ref{nod1})-(\ref{nod4}), when $l=1$, $y(x,0)=\sin \pi x$, and $y_t(x,0)=\sin \pi x$;  (a) $\mu=-0.1$; (b) $\mu=-0.2$;  (c) $\mu = -0.3$; (d) $\mu=-0.5$.}
\label{fig:Chapter4-03a}
\end{figure}

\newpage
\begin{figure}[!ht]
\begin{center}
\includegraphics[width=4.5 in,height=2.6 in ]{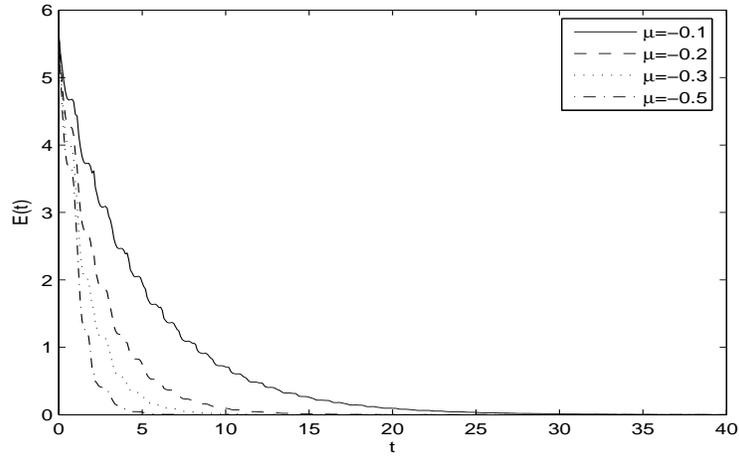}
\end{center}
\caption{The energy, $E(t)$, of the wave system without delay (\ref{nod1})-(\ref{nod4}), versus time for various values of $\mu$.}
\end{figure}

\begin{center}
\begin{figure}[!ht]
\centering
\subfigure[]{
\includegraphics[scale=0.5]{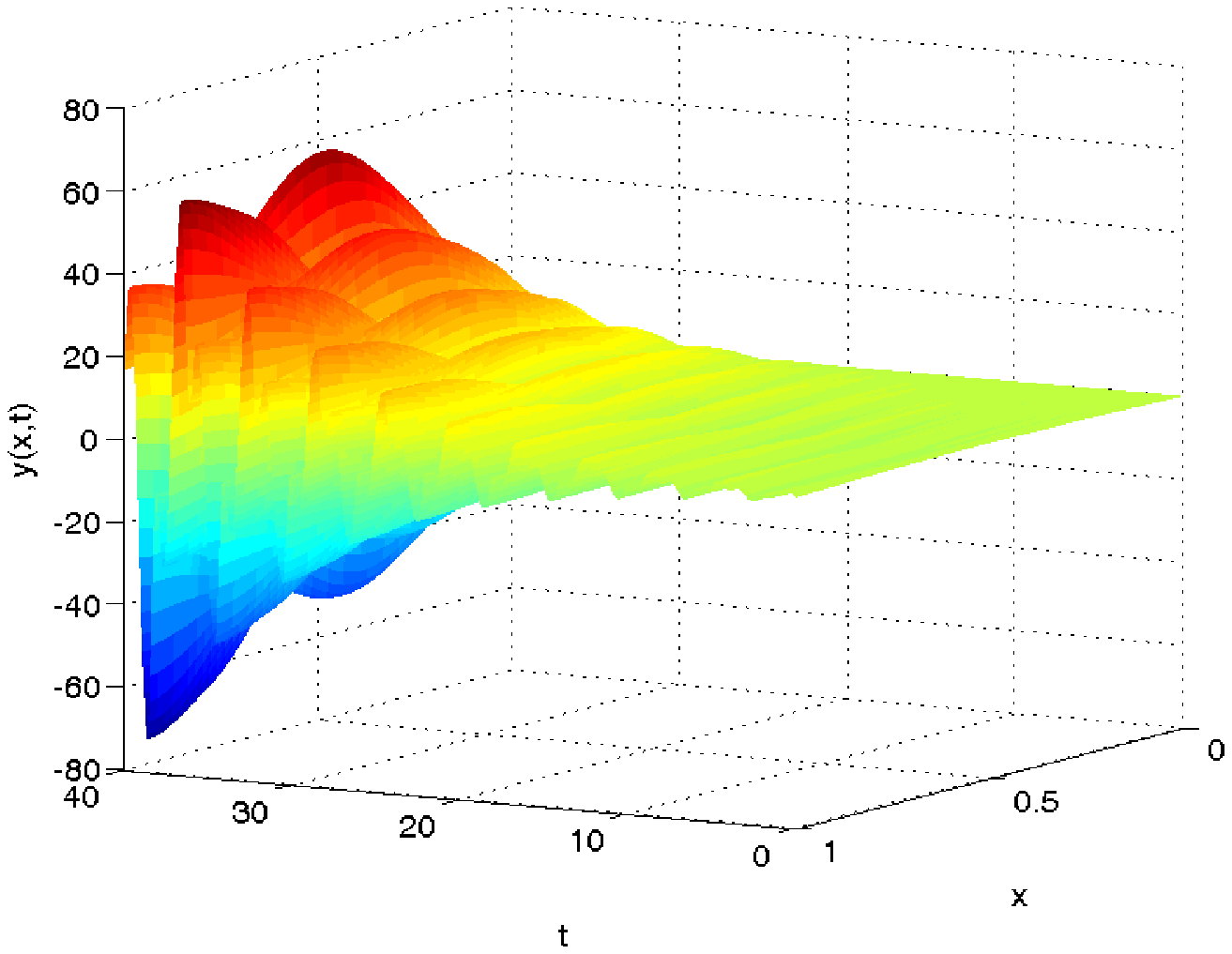}
\label{fig:subfiga2}
}
\subfigure[]{
\includegraphics[scale=0.5]{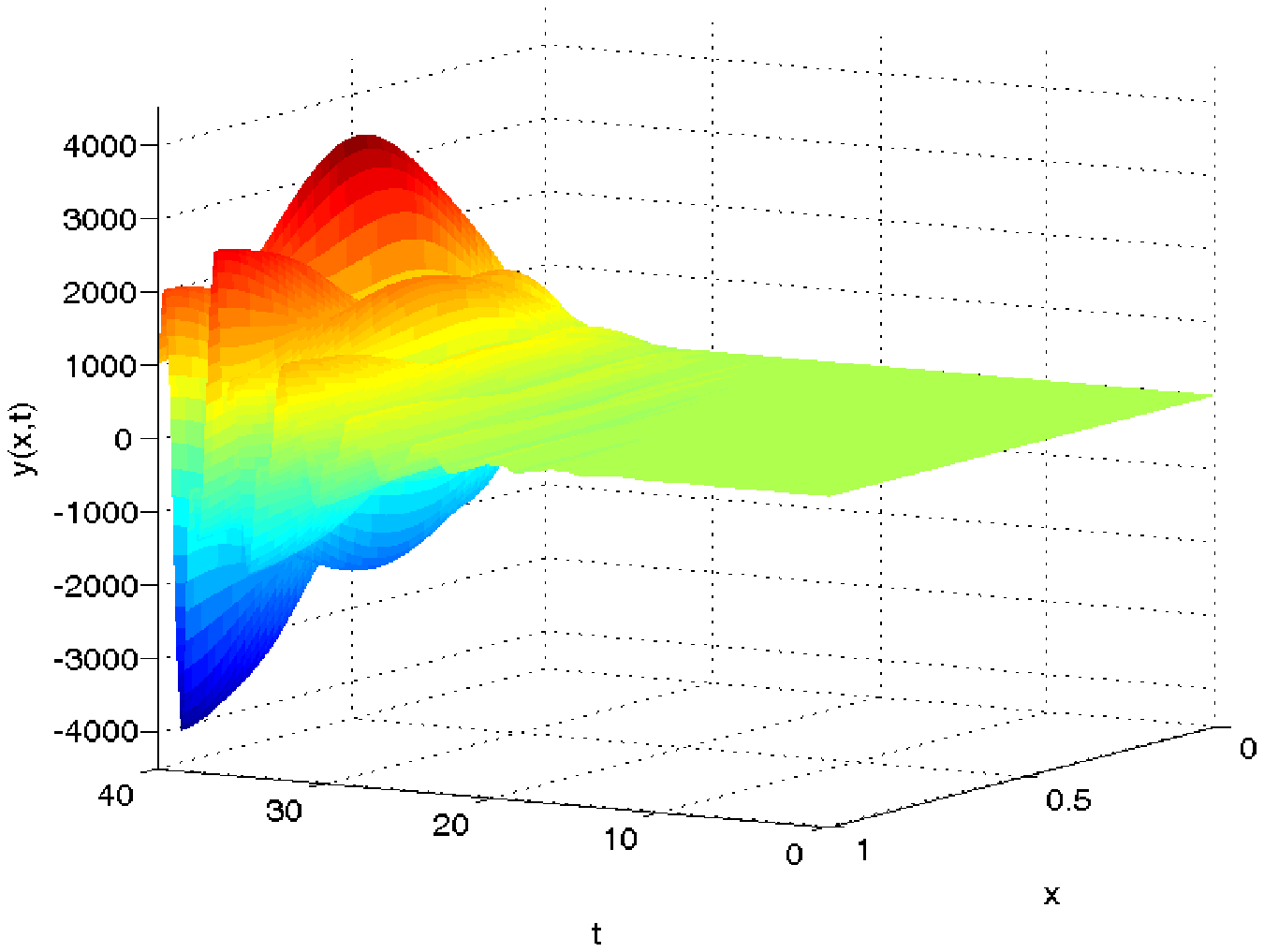}
\label{fig:subfigb2}
}
\subfigure[]{
\includegraphics[scale=0.5]{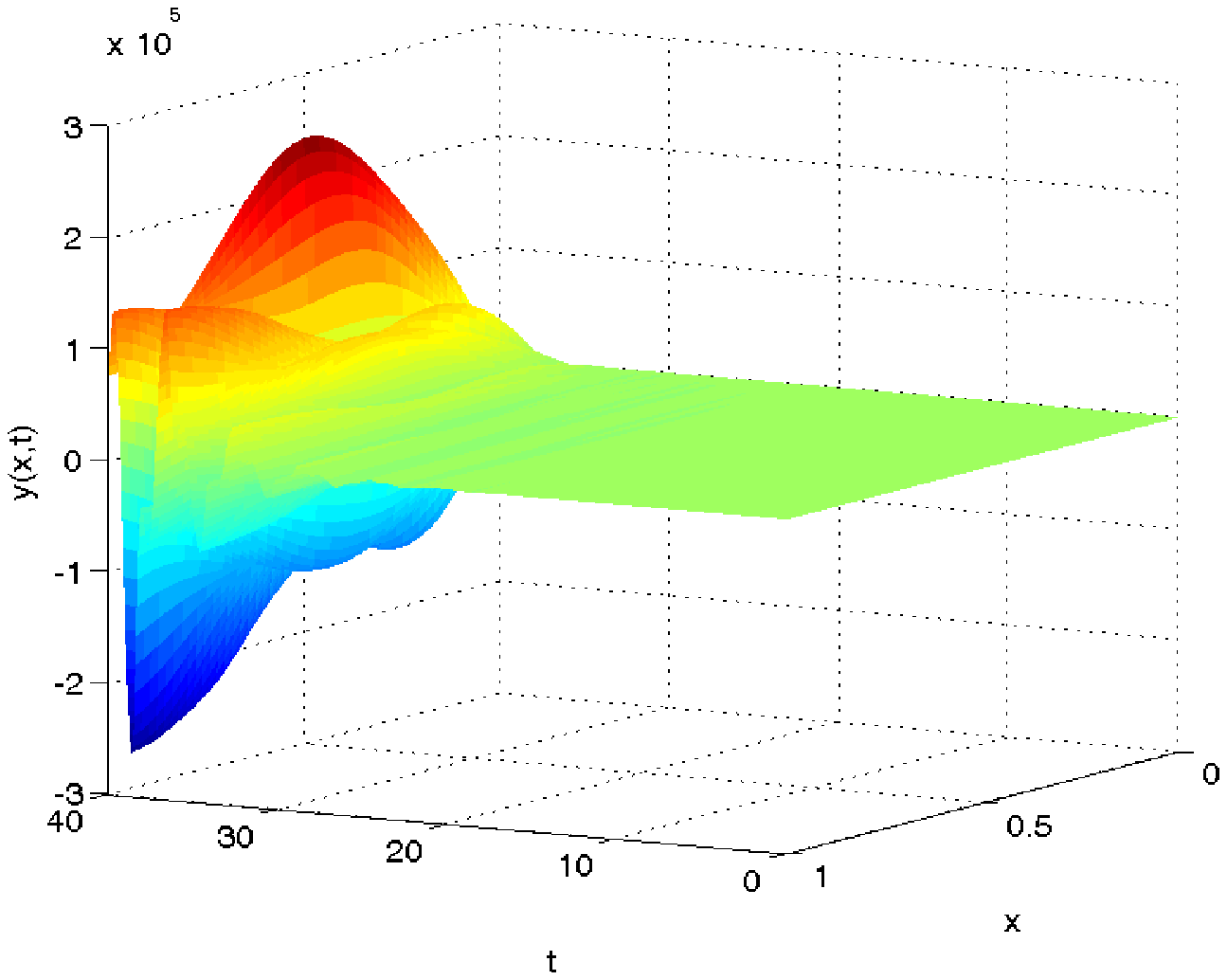}
\label{fig:subfigf2}
}
\subfigure[]{
\includegraphics[scale=0.5]{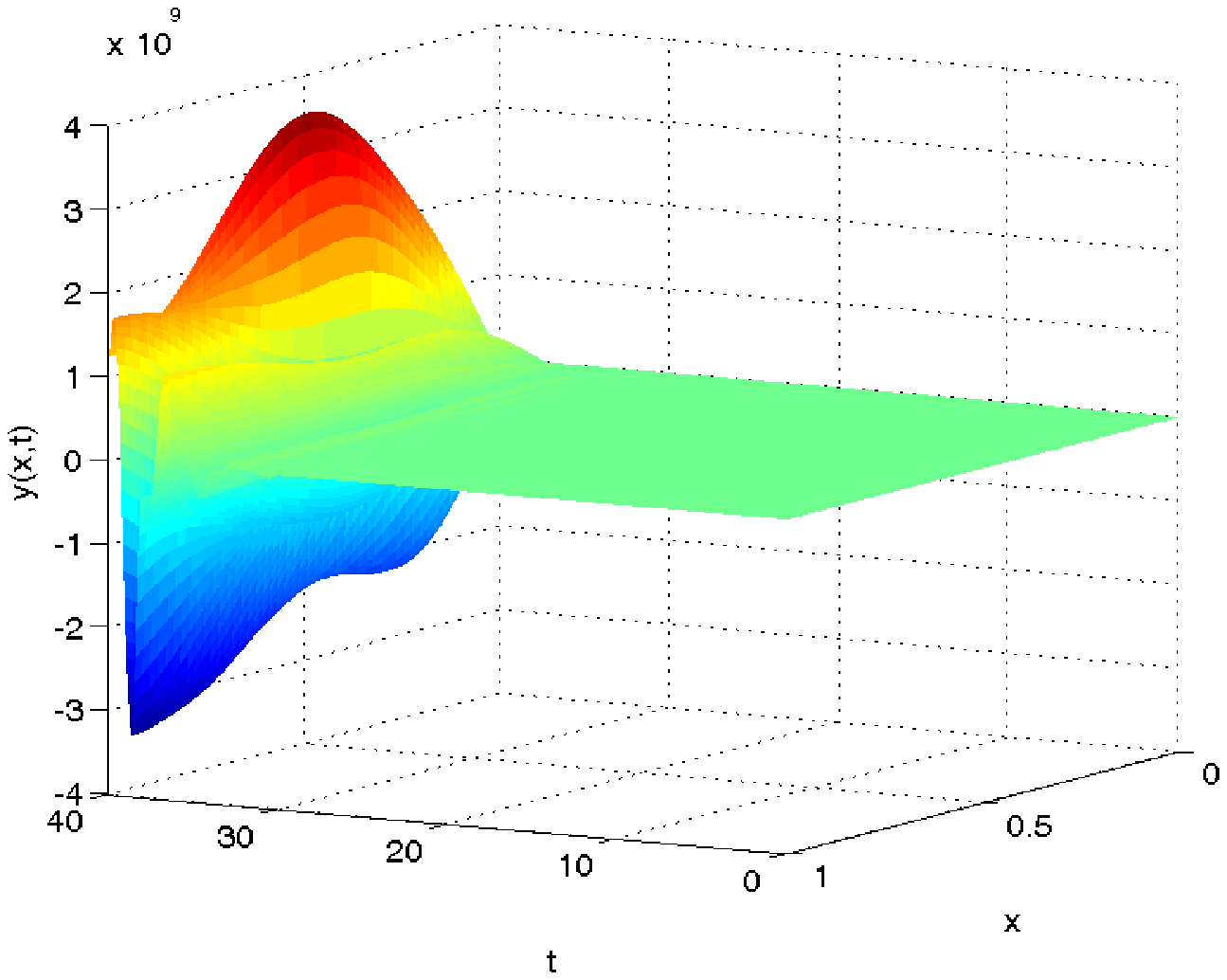}
\label{fig:subfige2}
}
\caption{A 3-d landscape of the dynamics of the wave equation without delay (\ref{nod1})-(\ref{nod4}), when $l=1$, $y(x,0)=\sin \pi x$, and $y_t(x,0)=\sin \pi x$;  (a) $\mu=0.1$; (b) $\mu=0.2$;  (c) $\mu = 0.3$; (d) $\mu=0.5$.}
\label{fig:Chapter4-03b}
\end{figure}
\end{center}

\newpage
\vspace*{1.5in}
\begin{center}
\begin{figure}[!ht]
\centering
\subfigure[]{
\includegraphics[scale=0.5]{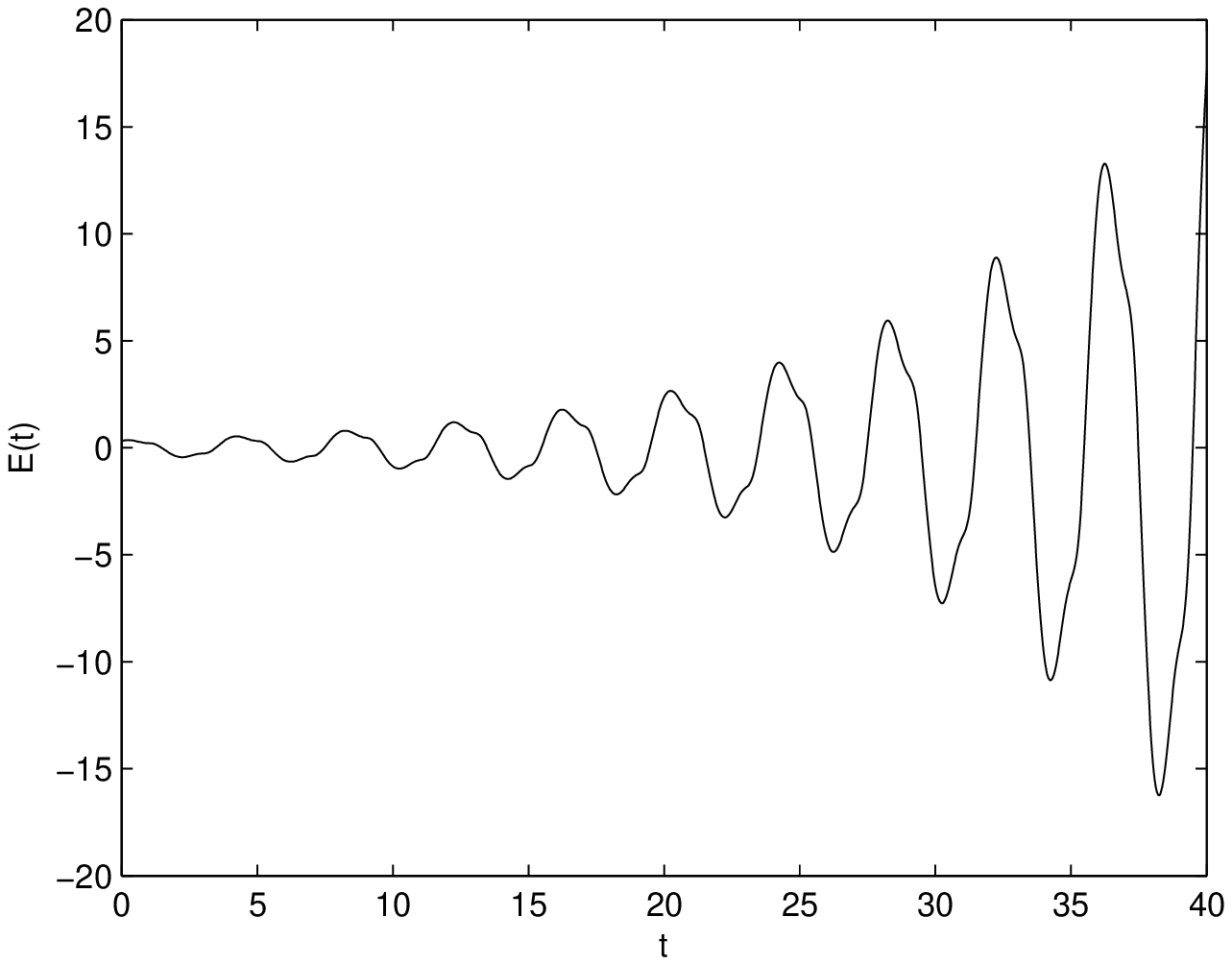}
\label{fig:subfiga3}
}
\subfigure[]{
\includegraphics[scale=0.5]{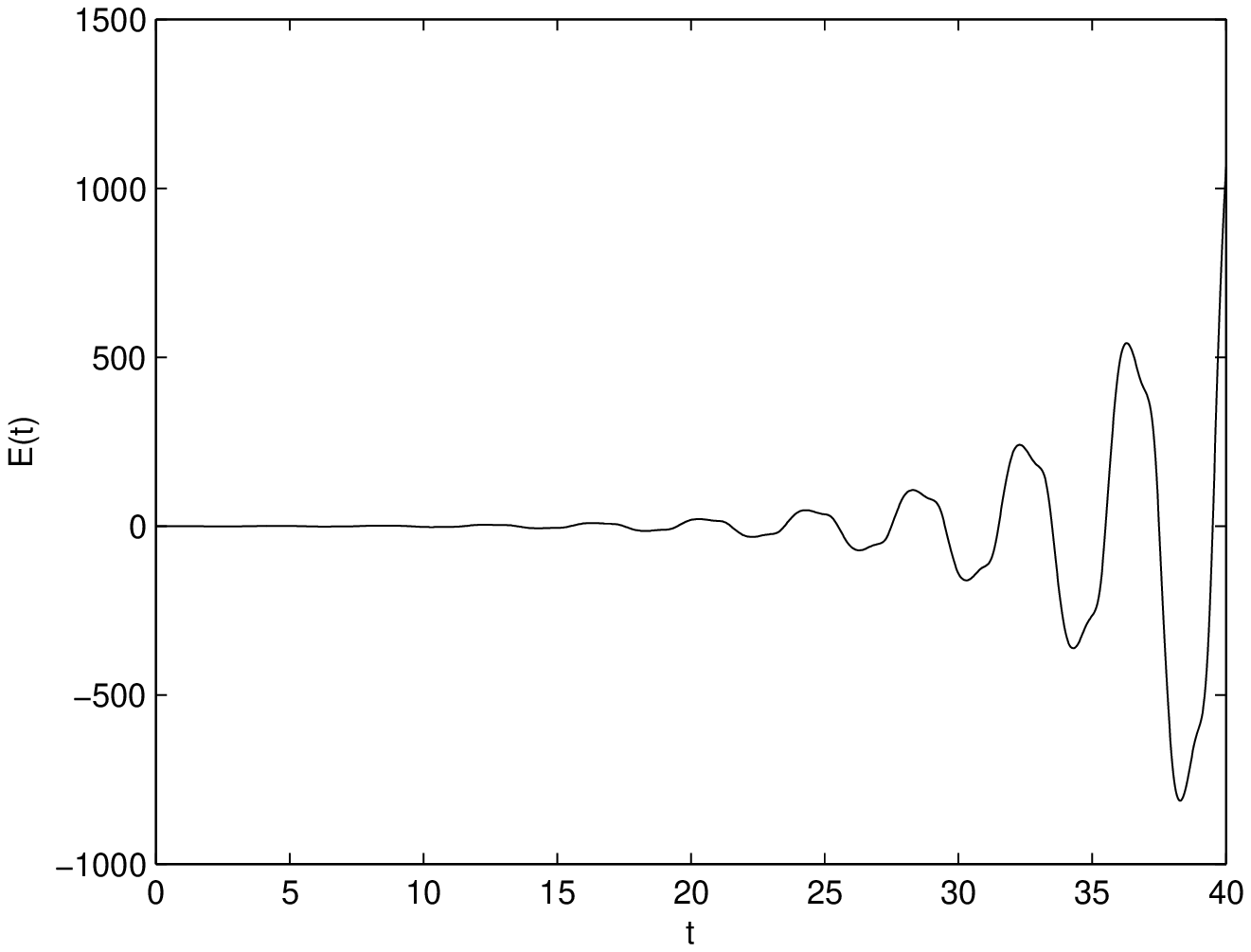}
\label{fig:subfigb3}
}
\subfigure[]{
\includegraphics[scale=0.5]{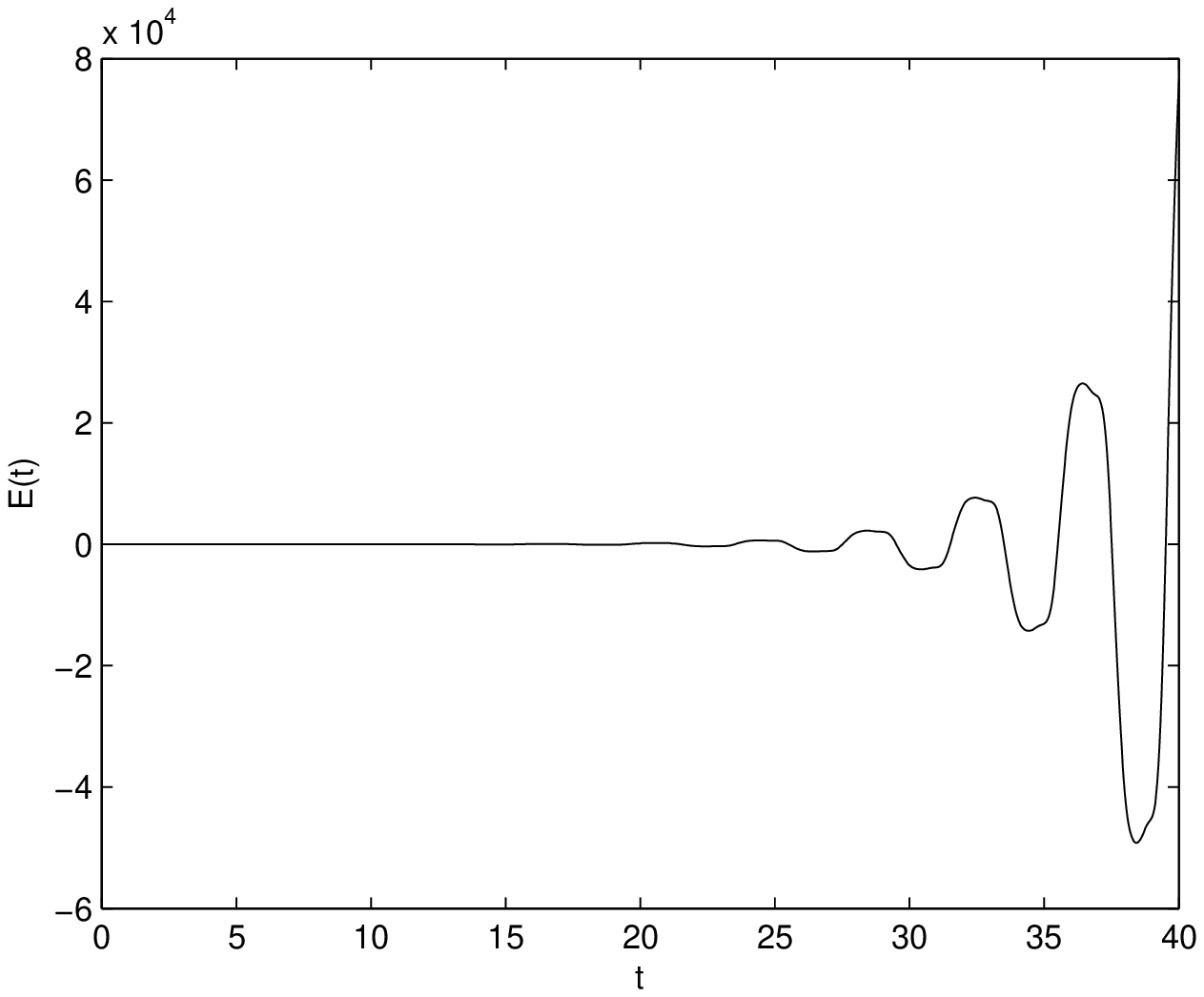}
\label{fig:subfigf3}
}
\subfigure[]{
\includegraphics[scale=0.5]{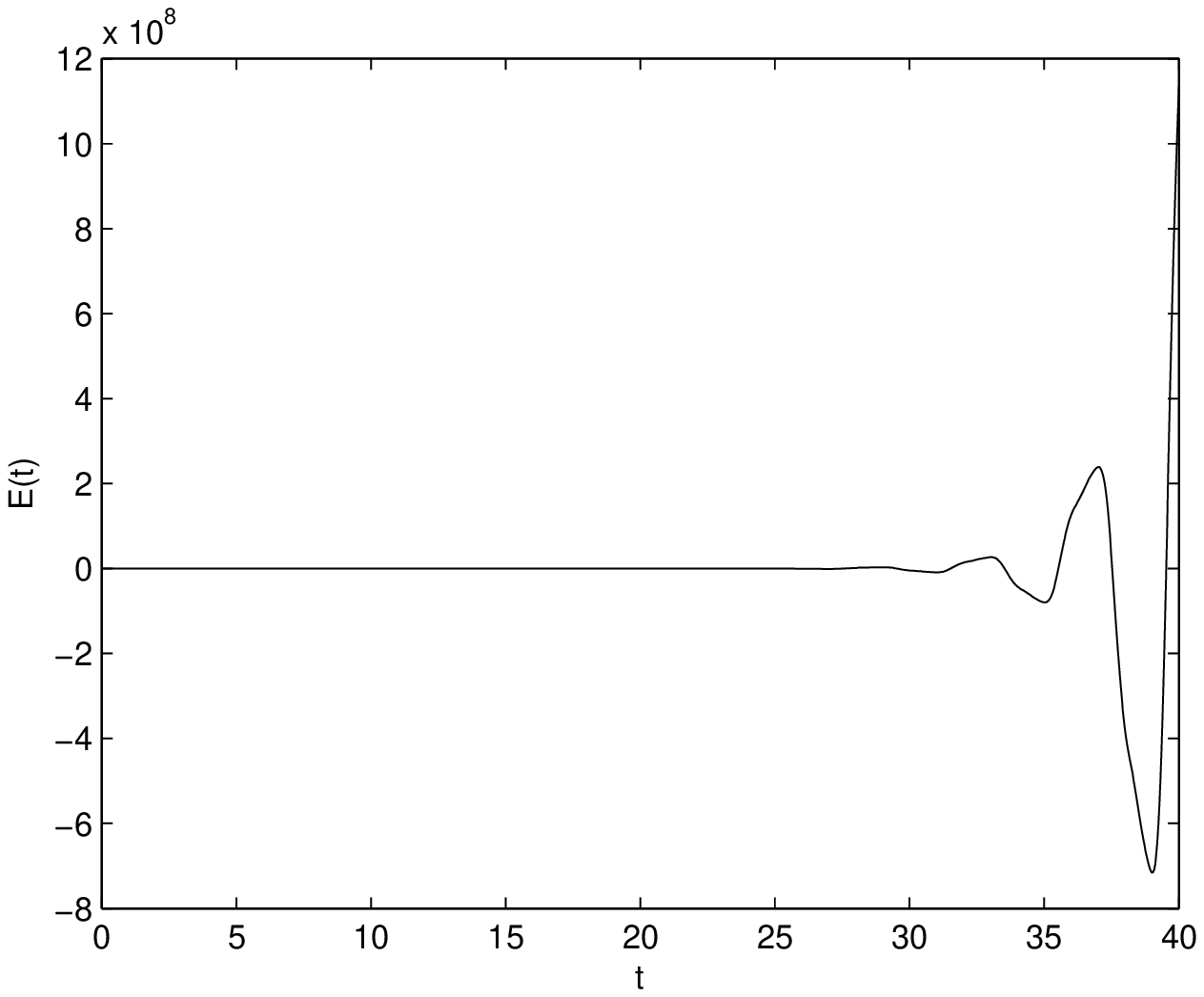}
\label{fig:subfige3}
}
\caption{The energy, E(t), of the dynamics of the wave equation without time delay (\ref{nod1})-(\ref{nod4}):  (a) $\mu=0.1$; (b) $\mu=0.2$;  (c) $\mu = 0.3$; (d) $\mu=0.5$.}
\label{fig:Chapter4-03c}
\end{figure}
\end{center}
\newpage
\vspace*{1.5in}
\begin{center}
\begin{figure}[!ht]
\centering
\subfigure[]{
\includegraphics[scale=0.5]{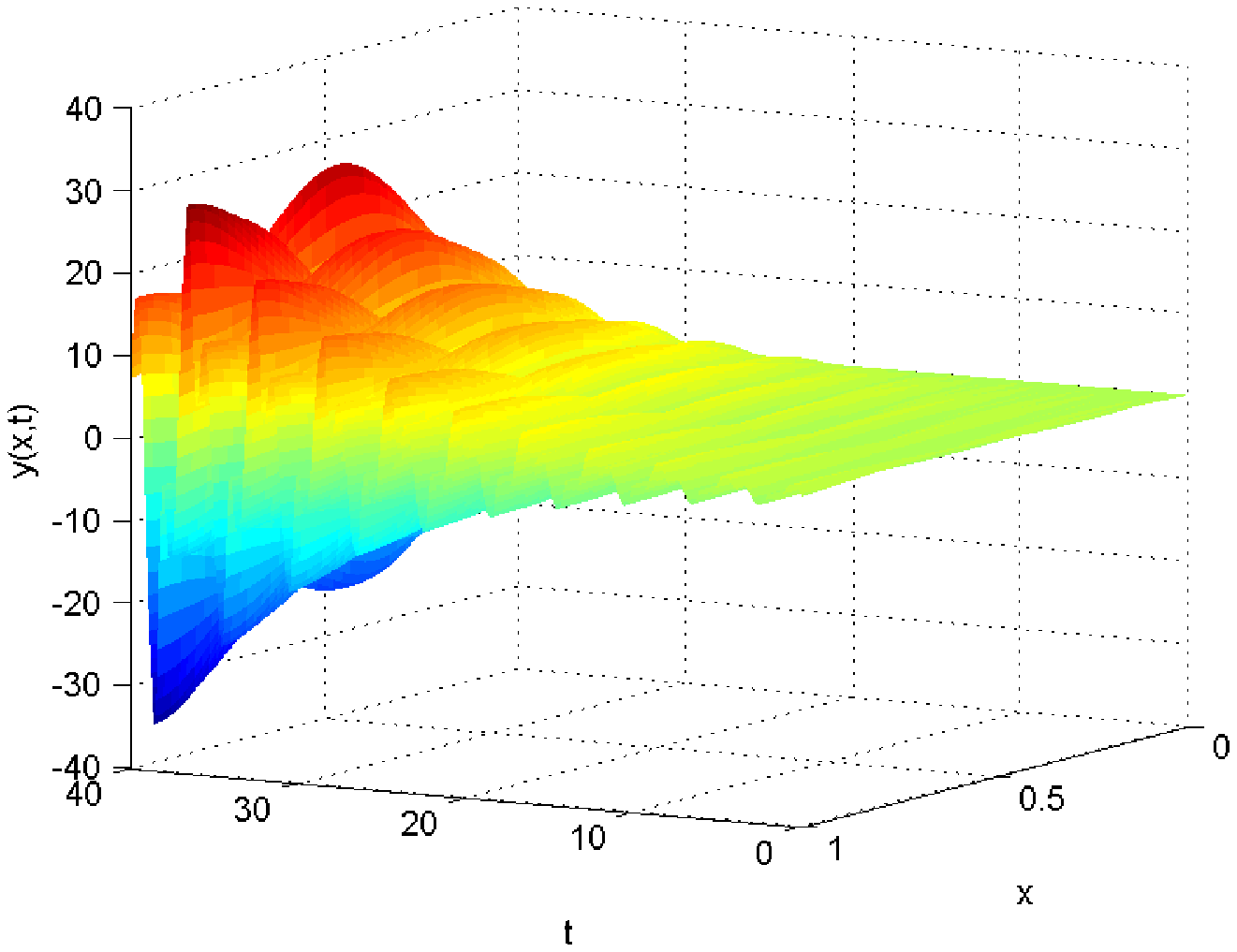}
\label{fig:subfiga4}
}
\subfigure[]{
\includegraphics[scale=0.5]{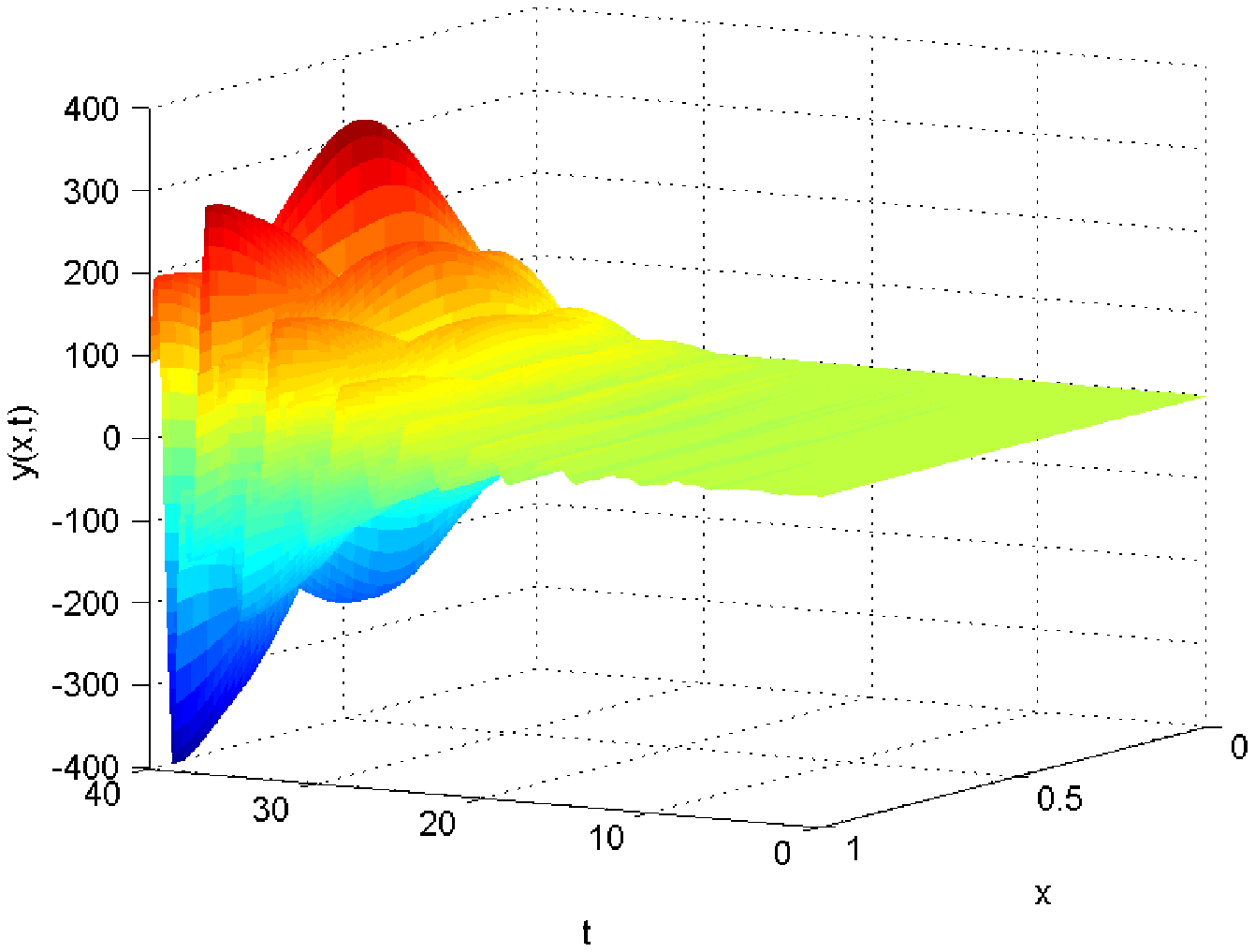}
\label{fig:subfigb4}
}
\subfigure[]{
\includegraphics[scale=0.5]{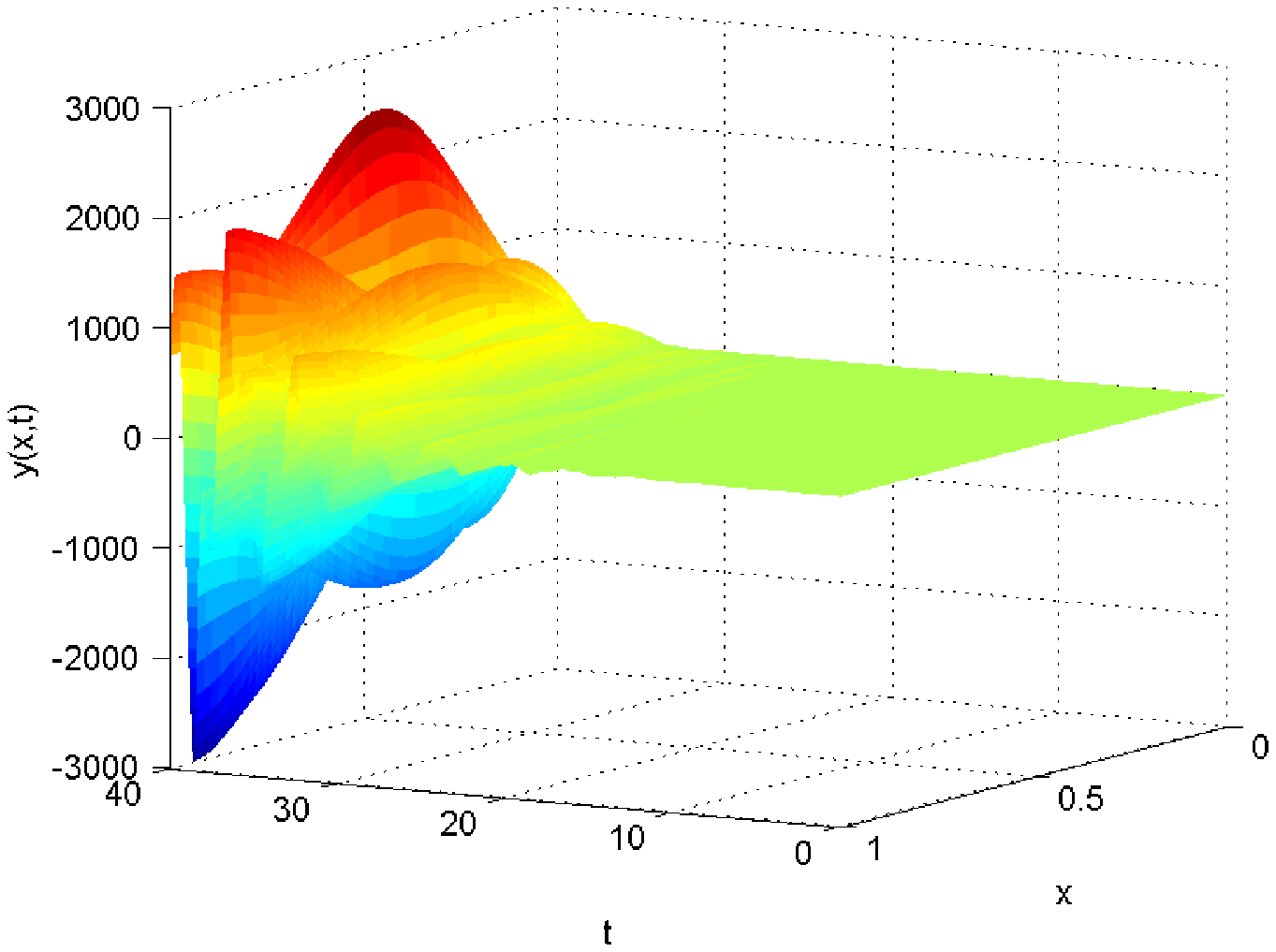}
\label{fig:subfigf4}
}
\subfigure[]{
\includegraphics[scale=0.5]{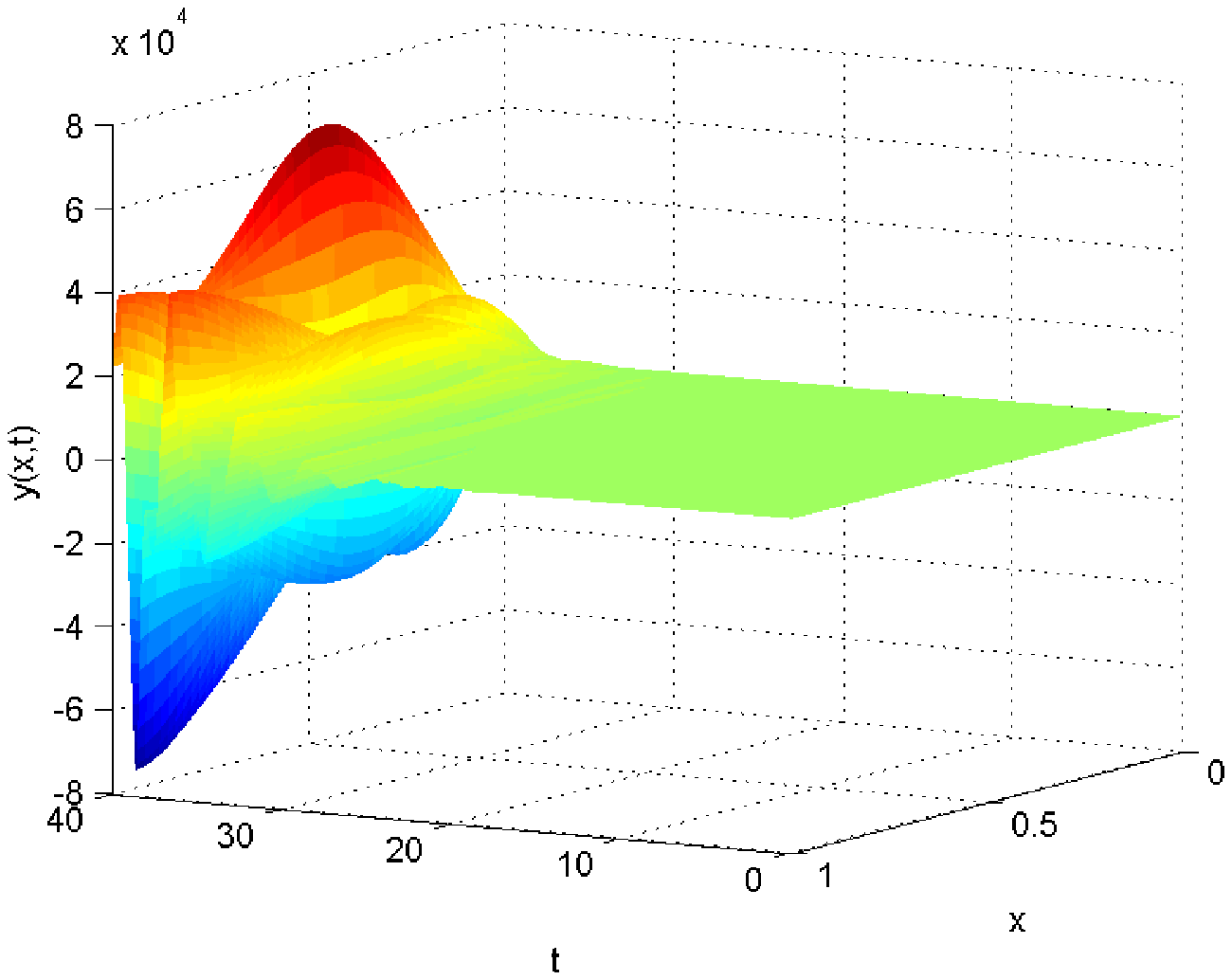}
\label{fig:subfige4}
}
\caption{A 3-d landscape of the dynamics of the wave equation with time delay $\tau=2\ell$, when $\ell=1$, $y(x,0)=\sin \pi x$, and $y_t(x,0)=\sin \pi x$;  (a) $\mu=-0.1$; (b) $\mu=-0.2$;  (c) $\mu = -0.3$; (d) $\mu=-0.5$.}
\label{fig:Chapter4-03d}
\end{figure}
\end{center}
\newpage
\vspace*{1.5in}
\begin{center}
\begin{figure}[!ht]
\centering
\subfigure[]{
\includegraphics[scale=0.5]{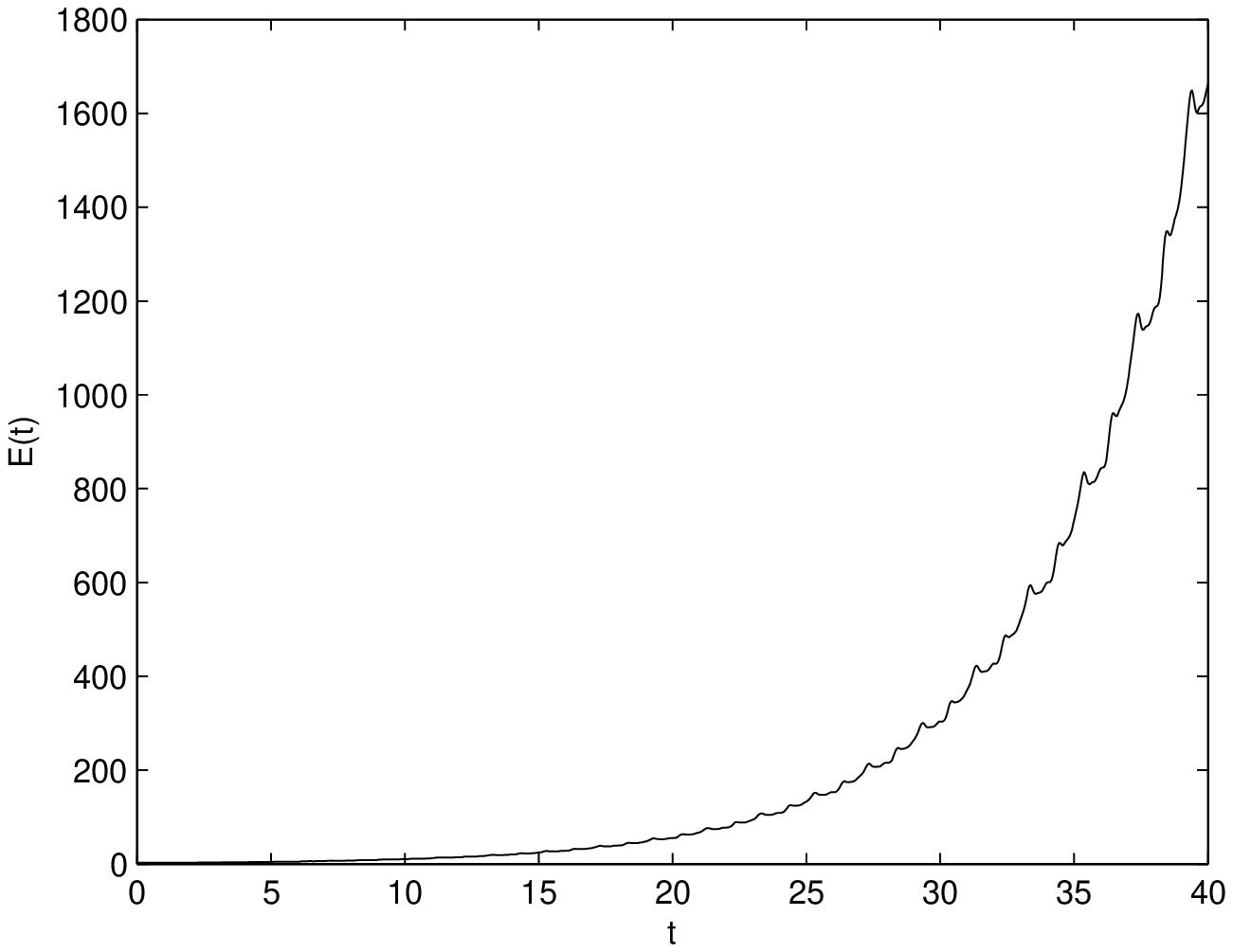}
\label{fig:subfiga5}
}
\subfigure[]{
\includegraphics[scale=0.5]{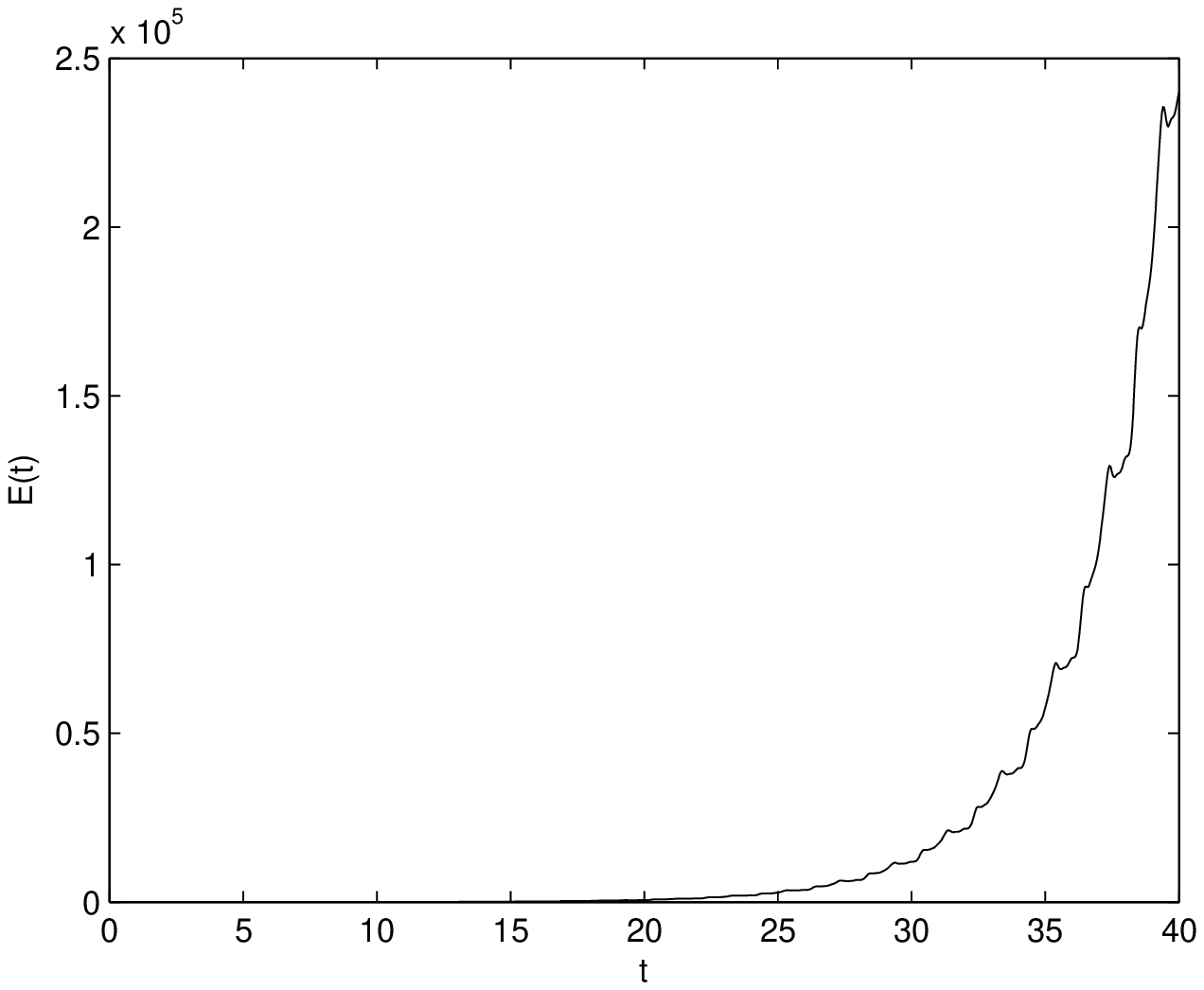}
\label{fig:subfigb5}
}
\subfigure[]{
\includegraphics[scale=0.5]{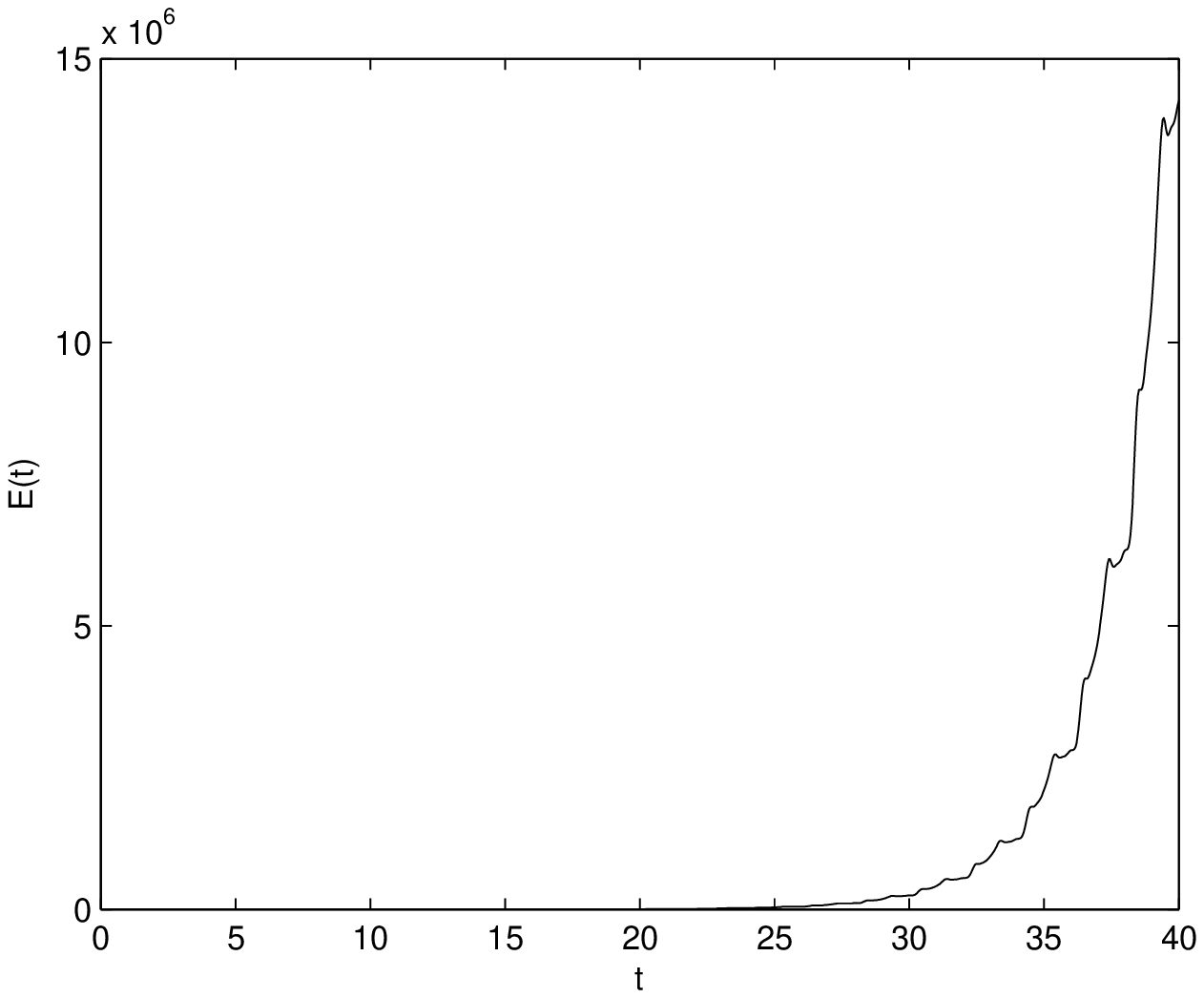}
\label{fig:subfigf5}
}
\subfigure[]{
\includegraphics[scale=0.5]{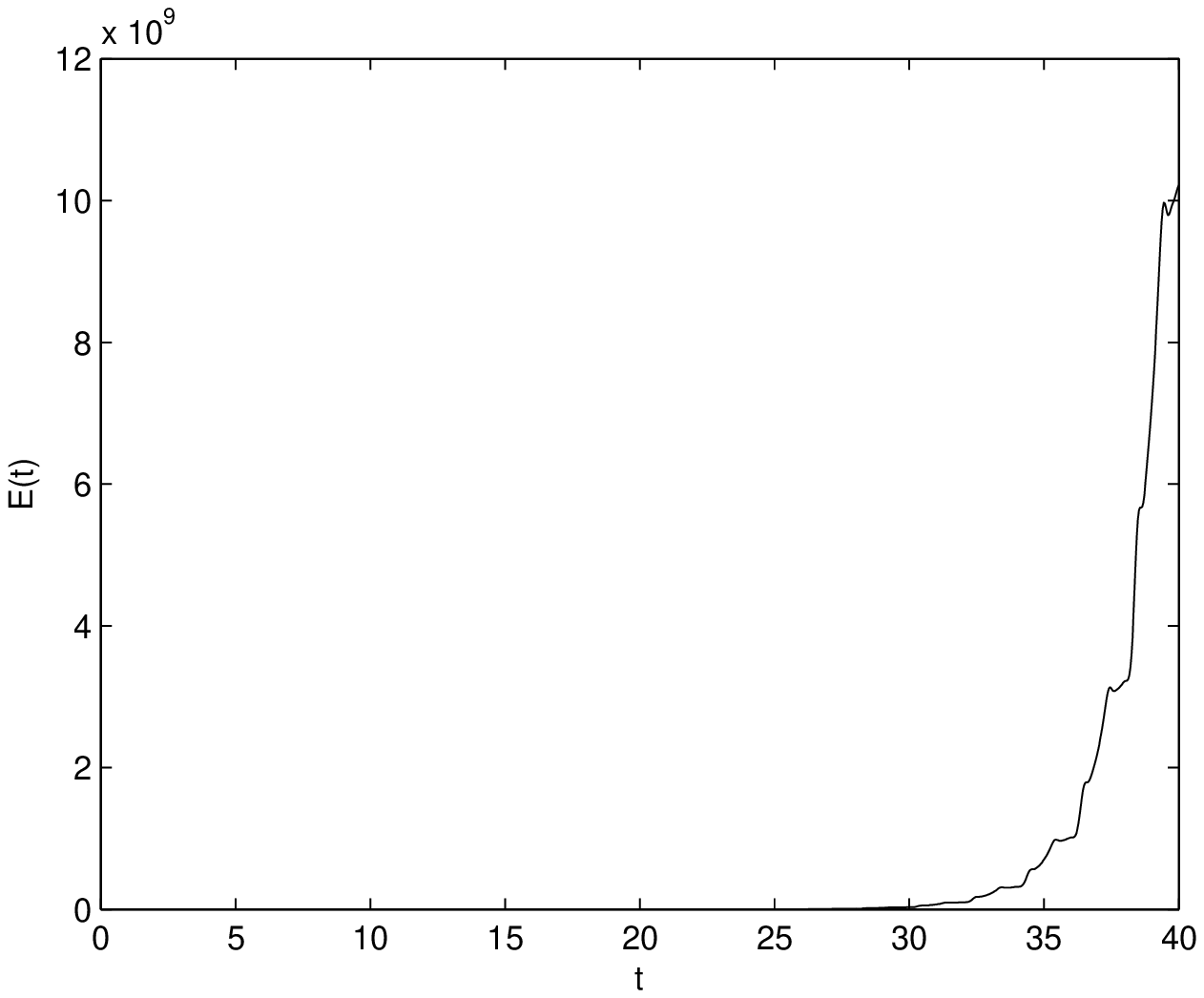}
\label{fig:subfige5}
}
\caption{The energy, E(t),  of the dynamics of the wave equation with time delay $\tau=2\ell$, when $\ell=1$, $y(x,0)=\sin \pi x$, and $y_t(x,0)=\sin \pi x$;  (a) $\mu=-0.1$; (b) $\mu=-0.2$;  (c) $\mu = -0.3$; (d) $\mu=-0.5$.}
\label{fig:Chapter4-03e}
\end{figure}
\end{center}
\newpage

\begin{figure}[!ht]
\centering
\subfigure[]{
\includegraphics[scale=0.5]{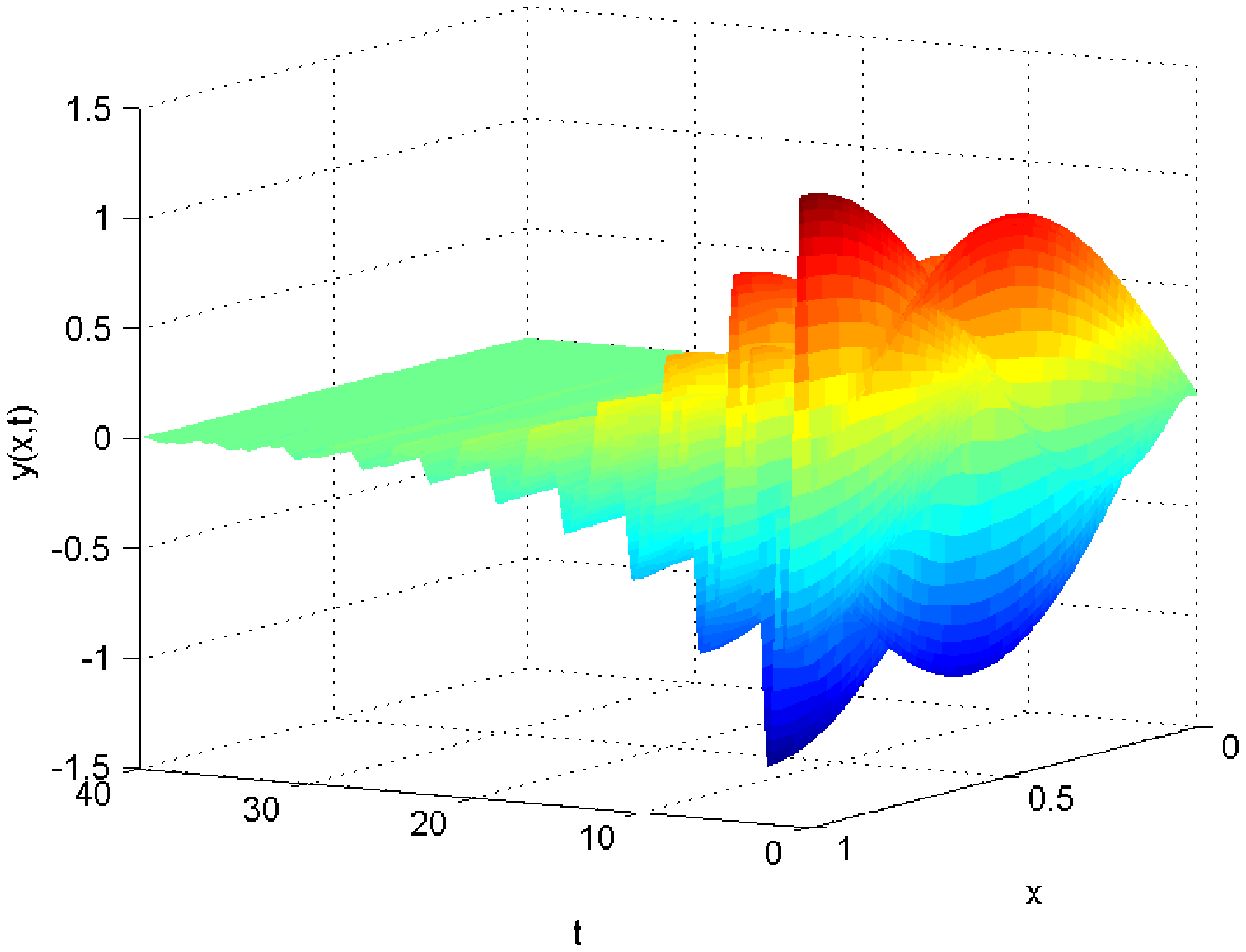}
\label{fig:subfiga6}
}
\subfigure[]{
\includegraphics[scale=0.5]{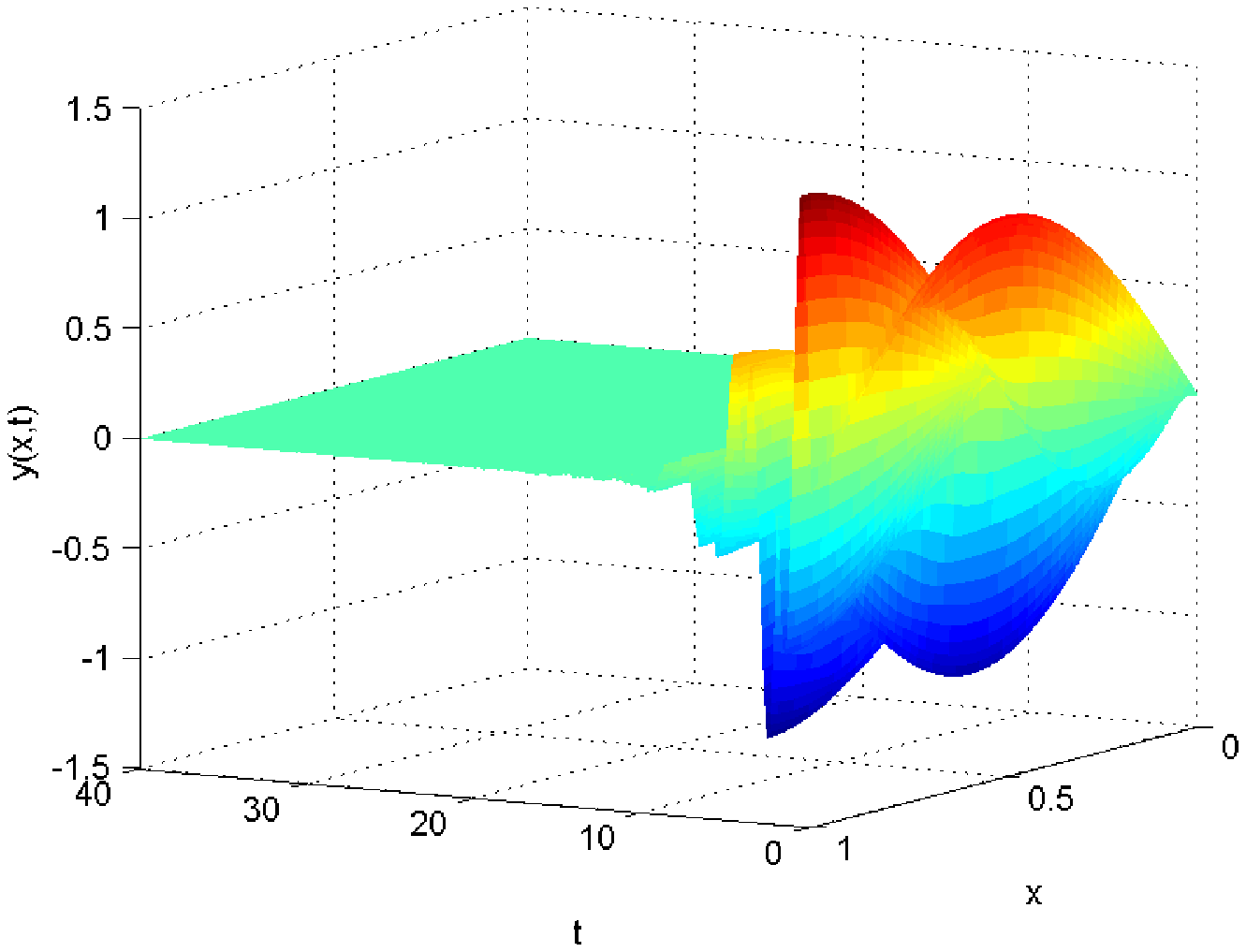}
\label{fig:subfigb6}
}
\subfigure[]{
\includegraphics[scale=0.5]{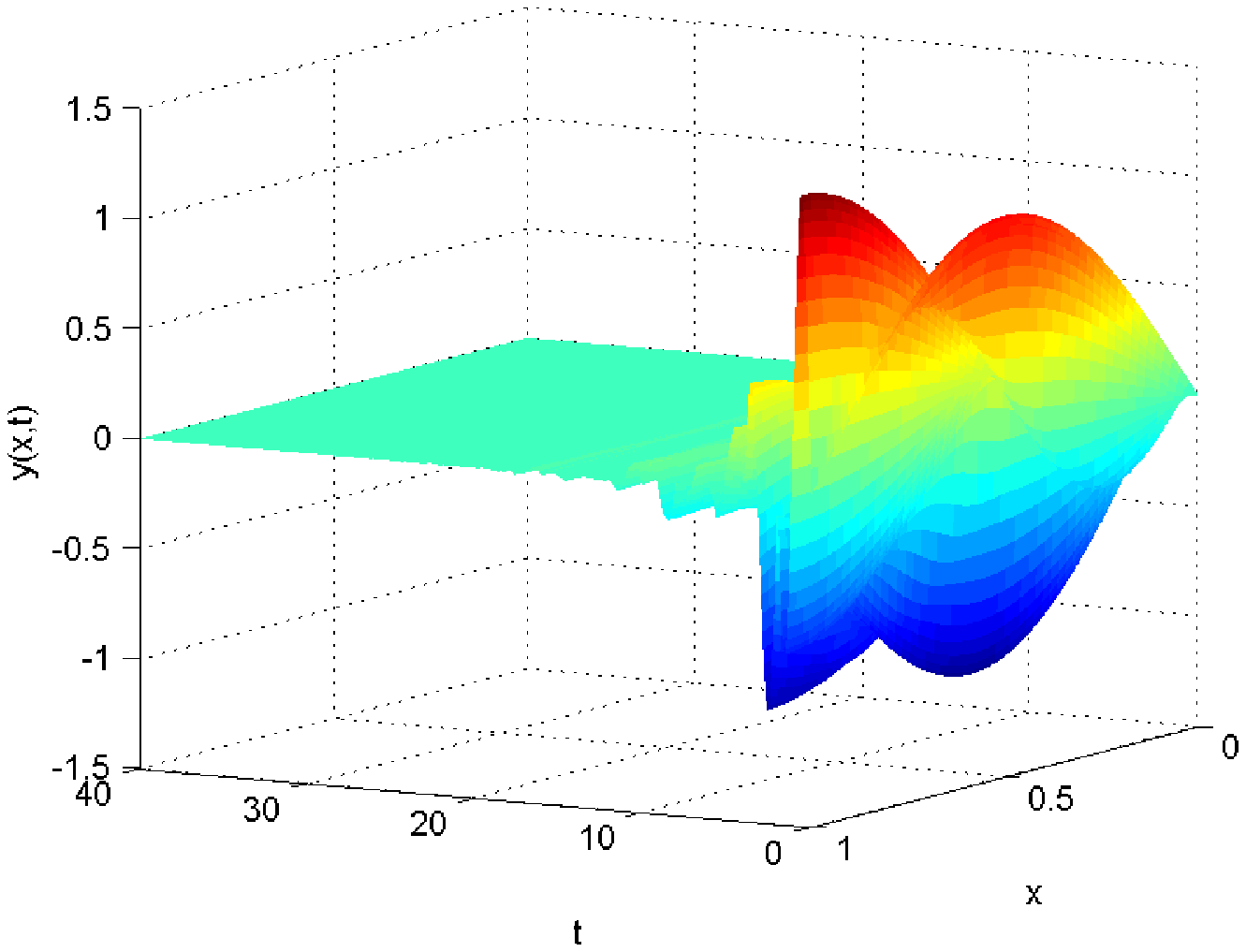}
\label{fig:subfigf6}
}
\subfigure[]{
\includegraphics[scale=0.5]{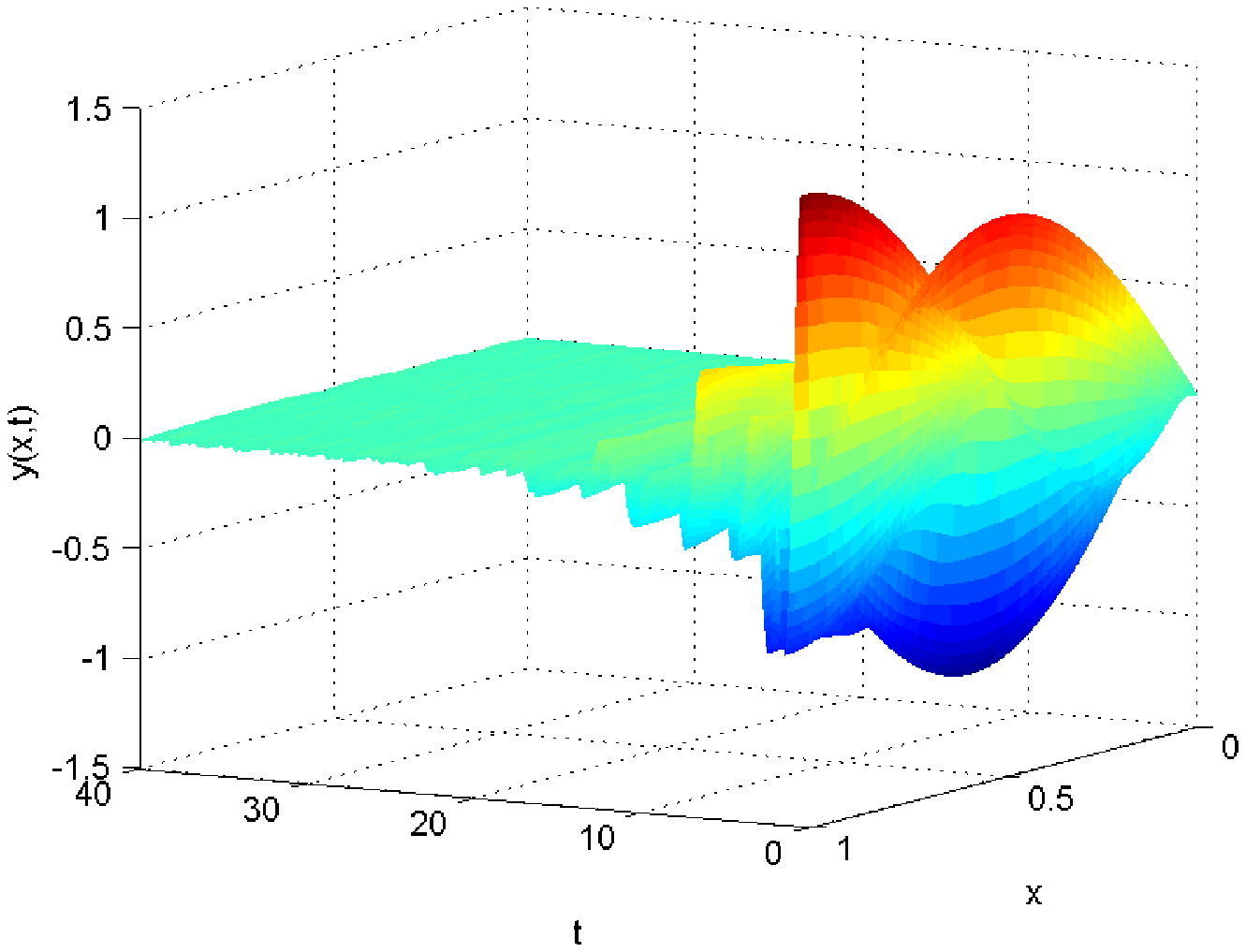}
\label{fig:subfige6}
}
\caption{A 3-d landscape of the dynamics of the wave equation with time delay $\tau=2\ell$, when $\ell=1$, $y(x,0)=\sin \pi x$, and $y_t(x,0)=\sin \pi x$;  (a) $\mu=0.1$; (b) $\mu=0.2$;  (c) $\mu = 0.3$; (d) $\mu=0.5$.}
\label{fig:Chapter4-03f}
\end{figure}

\begin{figure}[!ht]
\centering
\subfigure[]{
\includegraphics[scale=0.5]{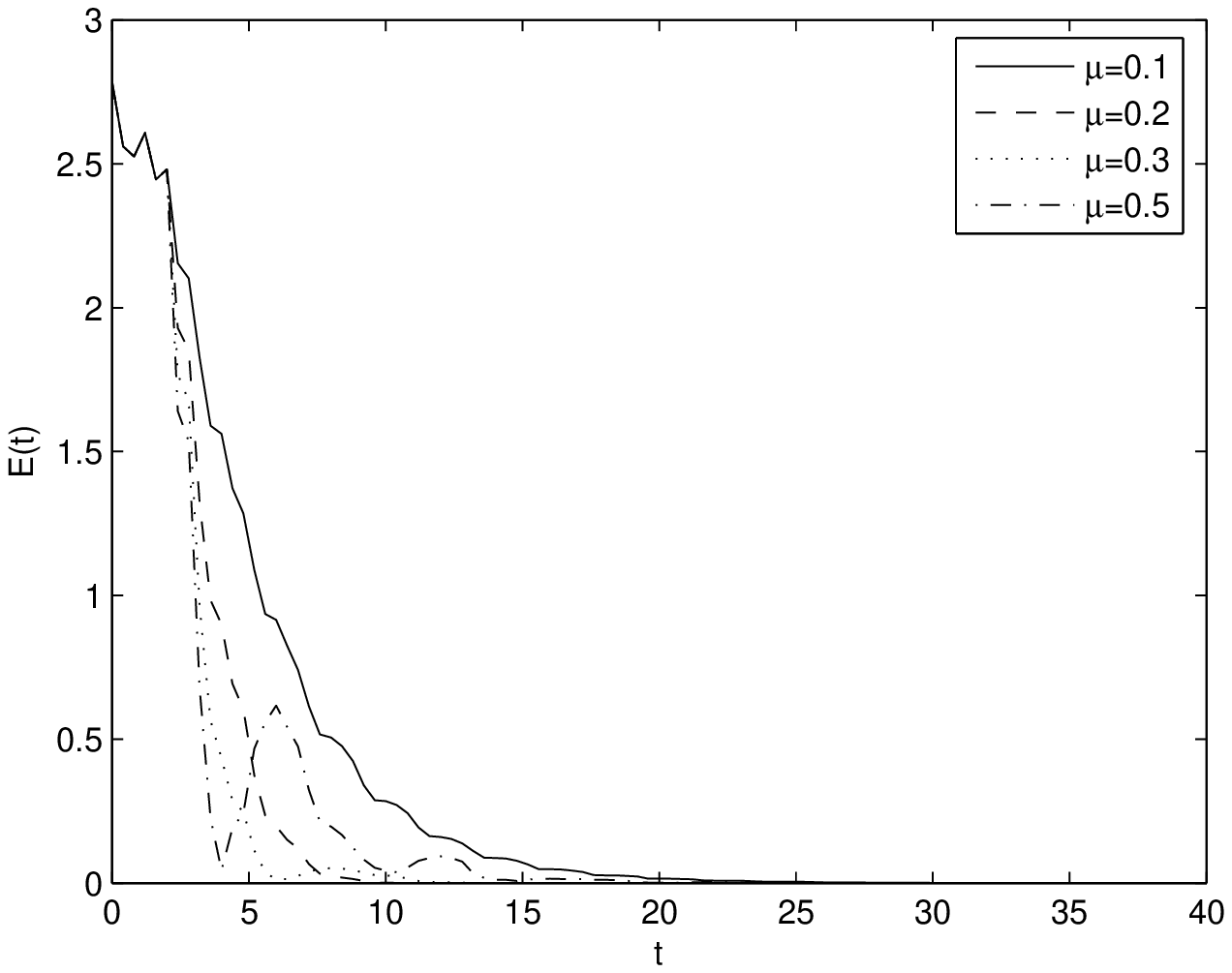}
\label{fig:subfiga7}
}
\subfigure[]{
\includegraphics[scale=0.5]{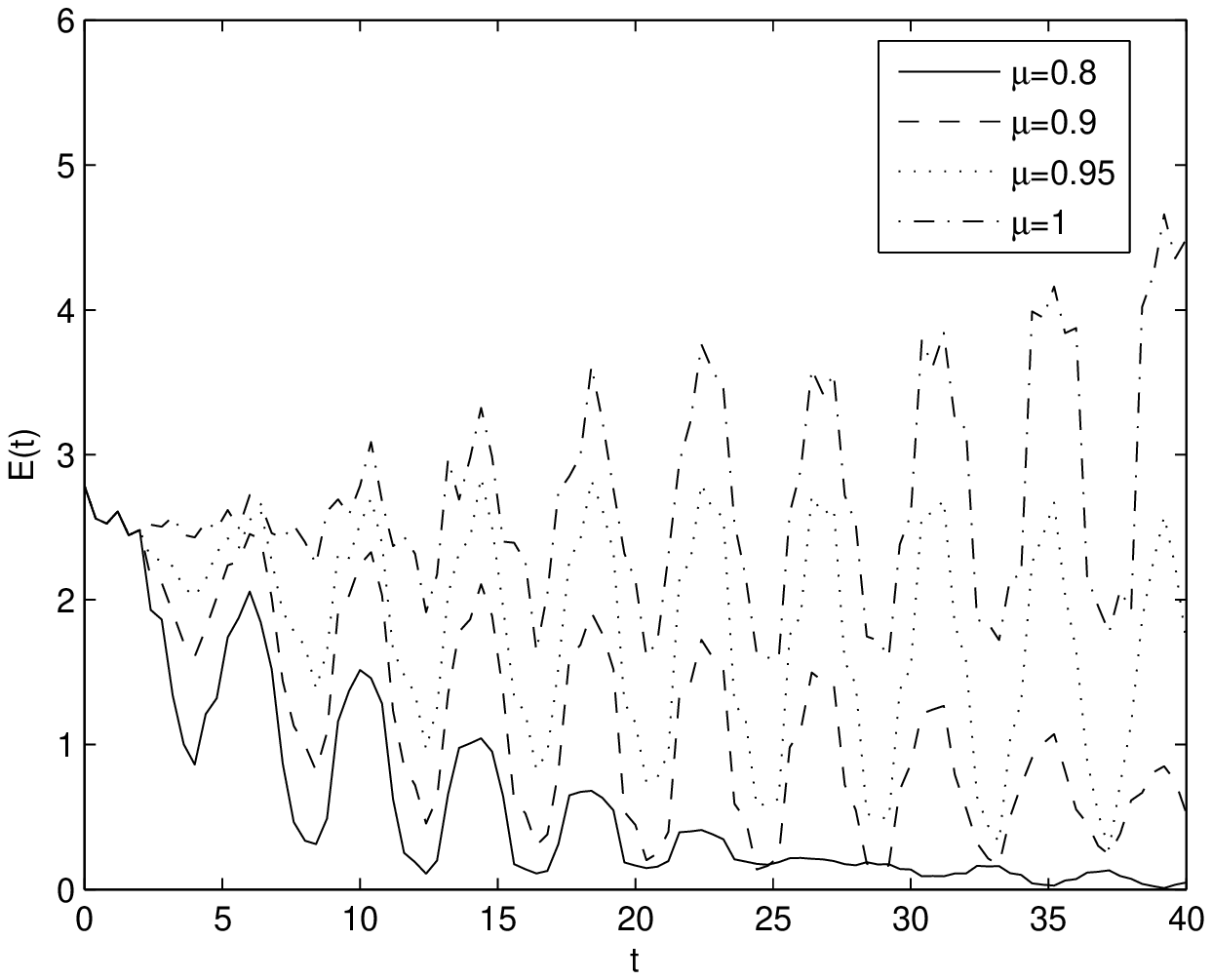}
\label{fig:subfigb7}
}
\caption{The energy, $E(t)$, versus time for various values of $\mu$ and when $\tau=2\ell$, with $\ell=1$; a) $\mu=0.1, \ldots, 0.5$; b) $\mu=0.8, \ldots, 1.0$; }
\end{figure}
\begin{figure}[!ht]
\begin{center}
\includegraphics[width=6.5 in,height=6 in ]{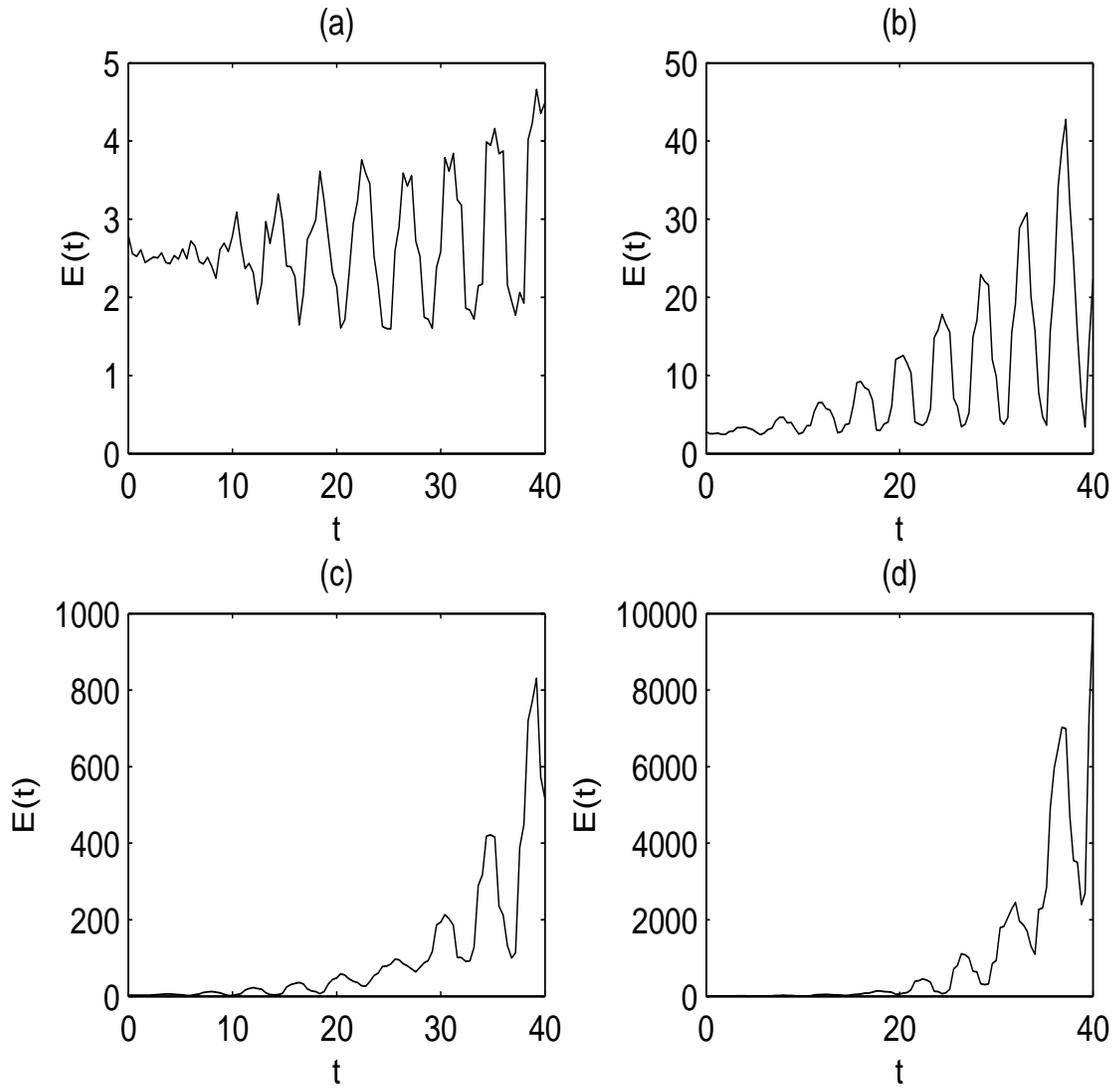}
\end{center}
\caption{The energy, $E(t)$, versus time for various values of $\mu$ and when $\tau=2\ell$, with $\ell=1$;  a) $\mu=1$; \,b) $\mu=1.1$; \, c) $\mu=1.3$; \, d) $\mu=1.5$.}
\end{figure}

\end{document}